\documentclass[11pt]{article}
\usepackage[dvipsnames,usenames]{color}
\usepackage{graphics}
\usepackage[utf8]{inputenc}  
\usepackage{makeidx}
\usepackage{amsmath,amsthm,amsfonts,amssymb}
\usepackage{hyperref}
\usepackage{tikz}
\usepackage{color}
\usepackage[title]{appendix}
\usepackage[affil-it, auth-lg]{authblk}

\newtheorem{proposition}{Proposition}

\newcommand\Pb{\mathbb{P}}
\newcommand\E{\mathbb{E}}

\newcommand\R{\mathbb{R}}
\newcommand\Mnr{\mathcal{M}_n}
\newtheorem{theorem}{Theorem}
\newtheorem{lemma}{Lemma}
\newtheorem{definition}{Definition}

\newcommand\pen{{\rm pen}}
\newcommand\penh{\widehat{\rm pen}}
\newcommand\penuv{\pen^{(uv)}}

\newcommand\hmuv{\hat{m}^{(uv)}}
\newcommand\hmkv{\hat{m}^{(kv)}}
\newcommand\N{\mathbb{N}}
\newcommand\Ld{\mathbb{L}^2([0,1])}
\newcommand{\EX}{\mathbb E_{\mathbf X}}
\newcommand{\Var}{{\rm Var}}

\newcommand{\tbkv}{\widetilde\beta^{(kv)}}
\newcommand{\tbuv}{\widetilde\beta^{(uv)}}

\newcounter{rem}
\newenvironment{remark}{\refstepcounter{rem}\\\textbf{Remark~\therem: }}{\\}
\title{Non-asymptotic Adaptive Prediction in Functional Linear Models}
\author{\'Elodie Brunel}
\author{Andr\'e Mas}
\author{Angelina Roche\footnote{Corresponding author: angelina.roche@univ-montp2.fr}}
\affil{ I3M, Universit\'e Montpellier II }
\date{}
\begin{document}
\maketitle

\begin{abstract}
Functional linear regression has recently attracted considerable interest. Many works focus on asymptotic inference. In this paper we consider in a non asymptotic framework a simple estimation procedure based on functional Principal Regression. It revolves in the minimization of a least square contrast coupled with a classical projection on the space spanned by the $m$ first empirical eigenvectors of the covariance operator of the functional sample. The novelty of our approach is to select automatically the crucial dimension $m$ by minimization of a penalized least square contrast. Our method is based on model selection tools. Yet, since this kind of methods consists usually in projecting onto known non-random spaces, we need to adapt it to empirical eigenbasis made of data-dependent – hence random – vectors. The resulting estimator is fully adaptive and is shown to verify an oracle inequality for the risk associated to the prediction error and to attain optimal minimax rates of convergence over a certain class of ellipsoids. Our strategy of model selection is finally compared numerically with cross-validation. 
\end{abstract}
\vspace{0.2cm}

\section*{Introduction}
Functional data analysis has known recent advances in the past two decades, addressing simultaneously many fields of applications. We refer to Ferraty and Vieu~\cite{ferraty_nonparametric_2006} and Ramsay and Silverman \cite{ramsay_functional_2005} for detailed examples in medicine, linguistics and chemometrics and to Preda and Saporta~\cite{preda_pls_2005} for applications in econometrics. 

In this paper we suppose that the dependence between a real-valued response $Y$ and a functional predictor $X$ belonging to a Hilbert space ($\mathbb H$, $<\cdot,\cdot>$,$\|\cdot\|$) is given by the functional linear model, namely
\begin{equation}\label{modele}
  Y=<\beta,X>+\varepsilon,
\end{equation}
where $\varepsilon$ stands for a noise term with variance $\sigma^2$ and is independent of $X$ and $\beta\in\mathbb H$ is an unknown function to be estimated. In order to simplify the notations, the random variable $X$ is supposed to be centred as well, which means that the function $t\mapsto \E[X(t)]$ is identically equal to zero. 

By multiplying both sides of Equation~(\ref{modele}) by $X(s)$ and taking the expectation, we see easily that the function $\beta$ is solution of
\begin{equation}\label{def:Gamma}
\Gamma\beta:=\E\left[<\beta,X>X(\cdot)\right]=\E\left[YX\right]=:g,
\end{equation}  
where $\Gamma$ is the covariance operator associated to the functional predictor $X$. Equation~(\ref{def:Gamma}) is known to be an ill-posed inverse problem (see Engl \textit{et al.}~\cite[Chapter 2.1]{engl_inverse_prblems_1996}).

The literature on the functional linear model is wide and numerous estimation procedures exist. A first method consists in minimizing a least square criterion subject to a roughness penalty. For instance, Li and Hsing~\cite{li_rates_2007} proposed an estimation procedure by minimization of such a criterion on periodic Sobolev spaces, Crambes \textit{et al. }\cite{crambes_smoothing_2009} generalized the well-known smoothing-spline estimator used in univariate nonparametric regression. Another approach is based on dimension reduction: this consists in approximating the regression function $\beta$ by projection onto finite-dimensional spaces. Those spaces are usually obtained by taking the first components of a basis of $\mathbb H$. Some authors considered projection onto fixed basis, such as B-spline basis (Ramsay and Dalzell \cite{ramsay_tools_1991}) or general orthonormal basis (Cardot and Johannes~\cite{cardot_thresholding_2010}). But the most popular method is Functional Principal Component Regression (FPCR), this consists in taking the random space spanned by the eigenfunctions associated to the largest eigenvalues of the empirical covariance operator:
\begin{equation}\label{def:Gamman}
\Gamma_n: f\in\mathbb H\mapsto\frac{1}{n}\sum_{i=1}^n<f,X_i>X_i.
\end{equation}
The resulting estimator is shown to be consistent, but its behaviour is often erratic in simulation studies, thus a smooth version by using splines  has been proposed by Cardot~\textit{et al.} \cite{cardot_spline_2003}. The FPCR estimator is shown to attain optimal rates of convergence for the risk associated to the prediction error over fixed curves $x$ (see Cai and Hall~\cite{cai_prediction_2006}) as well as for the $\mathbb{L}^2$-risk (see Hall and Horowitz~\cite{hall_methodology_2007}). 
 
All the proposed estimators rely on the choice of at least one tuning parameter (the smoothing parameter appearing in the penalized criterion or the dimension of approximation space) which influences significantly the quality of estimation. Optimal choice of such parameters depends generally on both unknown regularities of the slope function $\beta$ and the predictor $X$ (see e.g. \cite{cai_prediction_2006,crambes_smoothing_2009,cardot_thresholding_2010}) and the parameters are usually chosen in practice by cross-validation.  

 Until the recent work of Comte and Johannes~\cite{comte_adaptive_2010}, nonasymptotic results providing adaptive data-driven estimators were missing. Comte and Johannes~\cite{comte_adaptive_2010,comte_adaptive_2012} propose model selection procedures for the orthogonal series estimator introduced first by Cardot and Johannes~\cite{cardot_thresholding_2010}. In~\cite{comte_adaptive_2010}, they propose to select the dimension by minimization of a penalized contrast criterion under strong assumption of periodicity of the curve $X$ while in~\cite{comte_adaptive_2012} they define a dimension selection criterion by means of a stochastic penalized contrast emulating Lepski's method (see Goldenshluger and Lepski~\cite{goldenshluger_bandwidth_2011}) and do not require specific assumptions on the curve $X$. The resulting estimators are completely data-driven and achieve optimal minimax rates for general weighted  $\mathbb L^2$-risks. However, since both dimension selection criteria depend on weights defining the risk, these selection procedures do not address prediction error, which can be written as a weighted norm whose weights are the unknown eigenvalues of the covariance operator. 
 
 In the same context as Comte and Johannes~\cite{comte_adaptive_2010}, Brunel and Roche~\cite{brunel:hal-00651399} propose to estimate the slope function by minimizing a least square contrast on spaces spanned by the trigonometric basis. The dimension is selected by means of a penalized contrast. Their estimator is proved to attain the optimal minimax rate of convergence for the risk associated to the prediction error.   

Another approach is proposed by Cai and Yuan~\cite{cai_minimax_2012} carrying out reproducing kernel Hilbert spaces.  They develop a data-driven choice of the tuning parameter of the roughness regularization method (see e.g. Ramsay and Silverman~\cite{ramsay_functional_2005}). Their estimation procedure is shown to attain the optimal rate of convergence without the need of knowing the covariance kernel. Lee and Park~\cite{lee_sparse_2012} also suggest general variable selection procedures based on a weighted $L_1$ penalty under assumption of sparsity on the functional parameter $\beta$. Their estimator is shown to be consistent and to satisfy the oracle-property. 

In this paper, we propose an entirely data-driven procedure to select the adequate dimension for the classical FPCR estimator. The method proposed is based on model selection tools developed in a general context by Barron~\textit{et al.}~\cite{barron_risk_1999}, outlined by Massart~\cite{massart_concentration_2007}, and in a context of regression by Baraud~\cite{baraud_model_2000,baraud_model_2002}.  However, these tools are not meant to deal with estimators defined on random approximation spaces and thus have to be adapted. Section~\ref{estimation} is devoted to the description of estimation procedure. The resulting estimator is proved to satisfy an oracle-type inequality and to attain the optimal minimax rate of convergence for the risk associated to the prediction error for slope functions belonging to Sobolev classes in Section~\ref{oracle_vitesseCV}. In Section~\ref{simulation}, a simulation study is presented including a comparison with cross-validation. The proofs are detailed in Section~\ref{proofs} and in the Appendix.

\section{Definition of the estimator}
\label{estimation}

We assume that we are given an i.i.d. sample $(Y_i,X_i)_{i \geq 1}$ where the generic $Y$ is real and $X$ belongs to the Hilbert space $\mathbb H$. Thereafter, the Hilbert space is set to be $\mathbb H=\Ld$ equipped with its usual inner product $<\cdot,\cdot>$ defined by $<f,g>=\int_0^1f(t)g(t)dt$ but our method adapts to more general Sobolev spaces as well. We assumed above that $X$ is a centred random curve.

We recall that the theoretical covariance operator $\Gamma$ of $X$ defined by Equation~(\ref{def:Gamma}) in the introductory section is a selfadjoint trace class operator defined on and with values in $\Ld$. This means that the sequence of its eigenvalues denoted $(\lambda_j)_{j\geq 1}$ is positive and summable. The associated sequence of eigenfunctions is denoted by $(\psi_j)_{j\geq 1}$.

\subsection{Collection of models}

If the $(\psi_j)_{j\leq 1}$ were known an obvious choice would be to consider a model collection $(S_m)_m$ based on these eigenfunctions. Unfortunately this is not possible.
As the empirical covariance operator $\Gamma_n$ defined by Equation~(\ref{def:Gamman}) is selfadjoint too, there exists an orthonormal basis $(\hat\psi_j)_{j\geq 1}$ of $\Ld$ composed of eigenfunctions of $\Gamma_n$; we denote by $(\hat\lambda_j)_{j\geq 1}$ the associated eigenvalues arranged in decreasing order. Since $\Gamma_n$ is finished-rank, the $(\hat\lambda_j)_{j\geq 1}$ are necessarily null at least for $j > n$. The couples $(\hat\lambda_j, \hat\psi_j)_{j\geq 1}$ are the empirical counterparts of the $(\lambda_j, \psi_j)_{j\geq 1}$.
Dimension reduction based on functional Principal Component Analysis usually comes down to projecting the data on the space spanned by the $(\hat\psi_j)_{j\leq K}$ for some $K$. Our aim here is to shift to a model selection approach.


Let $\hat N_n$ be a random integer which will be defined later. For all $m\in\widehat{\Mnr}:=\{1,...,\hat N_n\}$, we define: 
 \begin{eqnarray*}
 \hat{S}_m&:=&\text{span}\{\hat{\psi}_1,...,\hat{\psi}_m\} 
\end{eqnarray*}

\noindent and the vector space $\hat{S}_m$ is an empirical counterpart of $S_m:=\text{span}\left\{\psi_1,...,\psi_m\right\}$. It is important to note that a major difference appears here. Classical model selection is carried out with fixed and known families of model. Here we handle random bases and this is the source of additional problems related to the convergence of (possibly random) projectors associated to these finite-dimensional spaces. Other difficulties come from the non-linear dependence between the coefficients of our estimator in the basis $(\hat\psi_j,....,\hat\psi_m)$ and the basis itself.

\subsection{Estimation on $\hat S_m$}

Introduce the following simple least square contrast:
\[\gamma_n(t):=\frac1n\sum_{i=1}^n(Y_i-<t,X_i>)^2.\]
Define $\hat g:=(1/n)\sum_{i=1}^nY_i X_i$ the cross-covariance between $Y$ and $X$ (which is also the theoretical counterpart of the function $g$ appearing in Equation~(\ref{def:Gamma})) and
\begin{equation}\label{def:hbetam}
 \hat{\beta}_m:=\sum_{j=1}^m\frac{<\hat g,\hat{\psi}_j>}{\hat{\lambda}_j}\hat{\psi}_j. 
\end{equation}
We can see easily that~(\ref{def:hbetam}) is the unique minimizer of the least square contrast $\gamma_n$ if $\hat\lambda_m>0$.  

From (\ref{modele}) the noise variance is as a parameter that we cannot dismiss since it appears in many computations. We distinguish two cases below.

\subsection{Model selection with known noise variance}
We suppose as a first step that the noise variance $\sigma^2$ is known. 
 
We set
\[\hat N_n:=\max\left\{N\in\N^*, N\leq20\sqrt{n/\ln^3(n)} \text{ and }\hat\lambda_N\geq \mathfrak{s}_n\right\},\]
 where $\mathfrak{s}_n:=\frac{2}{n^2}\left(1-\frac{1}{\ln^2 n}\right)$ ;
and its theoretical counterpart
\[N_n:=\max\left\{N\in\N^*, N\leq20\sqrt{n/\ln^3(n)}\text{ and }\lambda_N\geq n^{-2}\right\}.\] 

The introduction of an empirical maximal dimension is motivated by the need to ensure that the terms $\hat\lambda_j$ appearing in the definition of our estimator are not too small. 
 
We select the dimension $\hmkv\in\widehat{\Mnr}$ by minimizing the criterion
\begin{equation}\label{critvc}\text{crit}(m)=\gamma_n(\hat\beta_m)+\pen^{(kv)}(m)\end{equation}
with
\begin{equation}\label{penkv}
\pen^{(kv)}(m):=(1+\theta)\frac{\sigma^2}{n}m,
\end{equation}
where $\theta$ is a positive constant.
Then, we propose the following estimator of the function $\beta$
\begin{equation*}
\tbkv:=\hat\beta_{\hat m^{(kv)}}.
  \end{equation*}

\subsection{Model selection with unknown noise variance}
\label{subsec:selec_uv}
Let $\theta>4$ and $\delta>0$, we set
\[\hat N_n:=\max\left\{N\in\N^*, N\leq\min\{20\sqrt{n/ \ln^3(n)},n/\theta(1+2\delta)\} \text{ and }\hat\lambda_N\geq \mathfrak{s}_n\right\},\]
and 
\[N_n:=\max\left\{N\in\N^*, N\leq\min\{20\sqrt{n/\ln^3(n)},n/\theta(1+2\delta)\}\text{ and }\lambda_N\geq n^{-2}\right\},\]
the term $\sigma^2$ appearing in Equation~(\ref{penkv}) is replaced by the following estimator
\[\hat\sigma_m^2:=\frac{1}{n}\sum_{i=1}^n(Y_i-<\hat\beta_m,X_i>)^2=\gamma_n(\hat\beta_m). \]

The penalty becomes: 
\[\penh(m):=\theta(1+\delta)\hat{\sigma}_m^2\frac{m}{n},\]
and the selection criterion: 
\begin{equation}\label{critvi}\hat m^{(uv)}\in{\arg\min}_{m\in\widehat{\Mnr}}(\gamma_n(\hat\beta_m)+\widehat{\text{pen}}(m))={\arg\min}_{m\in\widehat{\Mnr}}\gamma_n(\hat\beta_m)\left(1+\theta(1+\delta)\frac mn\right).\end{equation}

We also denote 
\[\pen^{(uv)}(m):=\theta(1+\delta)\sigma^2\frac{m}{n},\]
the theoretical counterpart of $\widehat{\pen}(m)$.

Finally, we define the following estimator: 
\begin{equation*}
\tbuv:=\hat\beta_{\hat m^{(uv)}}.
\end{equation*}

In the sequel, when a property applies to both $\tbkv$ and $\tbuv$ we denote simply these estimators by $\widetilde\beta$, in that case we will denote also, $\hat m^{(kv)}$ and $\hat m^{(uv)}$ by $\hat m$ and $\pen^{(kv)}(m)$ and $\pen^{(uv)}(m)$ by $\pen(m)$.

\section{Main results}
\label{oracle_vitesseCV}
In this section we derive oracle-type inequalities and uniform bounds for the risk associated to the prediction error. The prediction error of an estimator $\hat\beta$ (see e.g.~\cite{crambes_smoothing_2009,cardot_thresholding_2010}) is defined by
\begin{equation}\label{prediction_error}
\E\left[\left(\hat{Y}_{n+1}-\E[Y_{n+1}|X_{n+1}]\right)^2|X_1,...,X_n\right]=\|\Gamma^{1/2}(\hat\beta-\beta)\|^2=:\|\hat\beta-\beta\|_\Gamma^2,
\end{equation}
where $\hat Y_{n+1}:=\int_0^1\hat\beta(t)X_{n+1}(t)dt$.
We suppose that $\lambda_j>0$ for all $j\geq 1$, which implies that the quantity~(\ref{prediction_error}) defines a norm on $\Ld$ denoted by $\|\cdot\|_\Gamma^2$. This condition is necessary for the model to be identifiable. Indeed, if there exists $j_0\geq 1$ such that $\lambda_{j_0}=0$, we have:
$$
0=\lambda_{j_0}\|\psi_{j_0}\|^2=<\Gamma\psi_{j_0},\psi_{j_0}>=\E\left[<X,\psi_{j_0}>^2\right],
$$
 and $<X,\psi_{j_0}>=0$ almost surely. By consequence, if the slope function $\beta$ satisfies Equation~(\ref{modele}), then any slope function of the form $\beta+c\psi_{j_0}$,  with $c\in\R$, satisfies also Equation~(\ref{modele}):   it is clearly impossible to identify the slope function with our sample in that case. However this condition is not sufficient, for more details on the problem of identifiability in functional linear models see Section 2 of Cardot~\textit{et al.}~\cite{cardot_spline_2003}. 

\subsection{Assumptions}

Recall that $(\lambda_j,\psi_j)_{j\geq 2}$ denote the eigenelements of the covariance operator $\Gamma$. We can control the risk under four assumptions: 
\begin{description}
\item[H1] There exists $p>4$ such that $\tau_p:=\E[|\varepsilon|^p]<+\infty$.
\item[H2] There exists $b>0$ such that, for all $l\in\N^*$, 
\[\sup_{j\in\N}\E\left[\frac{<X,\psi_j>^{2l}}{\lambda_j^l}\right]\leq l!b^{l-1}.\]
\end{description}

\begin{description}
\item[H3] For all $j\neq k$, $<X,\psi_j>$ is independent of $<X,\psi_k>$.
\item[H4] There exists a constant $\gamma>0$ such that the sequence $\left(j\lambda_j\ln^{1+\gamma}(j)\right)_{j\geq 2}$ is decreasing. 
\end{description}

Assumption~\textbf{H1} is standard in regression. Assumption~\textbf{H2} is necessary to apply exponential inequalities. Assumption~\textbf{H3} is also classical and we know from the Karhunen-Loeve decomposition of $X$ that it is true for $X$ a Gaussian process (see~\cite[Section 1.4]{ash_topics_1975}). Moreover, note that for every general random variables $X\in\Ld$, the random variables $<X,\psi_j>$ and $<X,\psi_k>$ are uncorrelated since if $j\neq k$, $\E[<\psi_j,X_i><\psi_k,X_i>]=<\Gamma\psi_j,\psi_k>=0$.  The assumption on the sequence $\left(j\lambda_j\ln^{1+\gamma}(j)\right)_{j\geq 2}$ allows to avoid more restrictive hypotheses about spacing control between eigenvalues as usually made frequently in the literature (see \cite{cai_prediction_2006}, \cite{hall_properties_2006}, \cite{hall_methodology_2007}).

In order to derive oracle-inequalities for the risk associated to the prediction error, we need to precise the decreasing rate of the sequence $(\lambda_j)_{j\geq 1}$. Usually in functional linear regression, this rate is supposed to be polynomial (see for instance \cite{cai_prediction_2006,crambes_smoothing_2009,cai_minimax_2012}) but more regular processes may be considered. That is the reason why, following Cardot and Johannes~\cite{cardot_thresholding_2010} or Comte and Johannes~\cite{comte_adaptive_2010}, we consider also exponential rates.
\begin{description}
\item[$(\mathbf P)$ Polynomial decrease] There exists two constants $a>1$ and $c_P\geq1$ such that, for all $j\geq 1$
\[c_P^{-1}j^{-a}\leq \lambda_j\leq c_Pj^{-a}.\]
\item[$(\mathbf E)$ Exponential decrease] There exists two constants $a>0$ and $c_E\geq 1$  such that for all $j\geq 1$
\[c_E^{-1}\exp(-j^a)\leq \lambda_j\leq c_E\exp(-j^a).\]
\end{description}

\subsection{Upper-bound on the empirical risk}
\label{subsec:control_empirical_risk}

We define an empirical semi-norm naturally associated to our estimation problem by 
\[\|f\|_{\Gamma_n}^2:=\|\Gamma_n^{1/2}f\|^2=\frac 1n\sum_{i=1}^n<f,X_i>^2, \text{ for all }f\in\Ld.\]

In a first step, in propositions \ref{risknvc} and \ref{risknvi}, we prove that our estimators verify an oracle type inequality for the risk associated to this semi-norm whatever the regularity of the slope function $\beta$ and the decreasing rate of the covariance operator eigenvalues are.

\subsubsection{Bound on the empirical risk with known noise variance}
\begin{proposition}\label{risknvc}Suppose that Assumption~\textbf{H1} is fulfilled, we have
\begin{equation*}
\mathbb E[\|\tbkv-\beta\|_{\Gamma_n}^2]\leq C\inf_{m\in\Mnr}\left\{\mathbb E[\|\beta-\hat\Pi_m\beta\|_{\Gamma_n}^2]+\pen^{(kv)}(m)\right\}+\frac{C'(\sigma^2+\|\beta\|^2)}{n}, \end{equation*}
with $C,C'>0$ depending only on $\theta$ and $p$ and $\hat\Pi_m$ the orthonormal projector onto $\hat S_m$.
\begin{proof}
We want here to take advantage of Corollary 3.1  in Baraud~\cite{baraud_model_2000} who provided a very similar result in the context of regression on a fixed design. Indeed Baraud considers a model $Y=s(x)+\epsilon$ where $Y$ is real, $x$ takes values in some measurable space and $s$ is a general mapping. Conditioning our functional linear model with respect to $\mathbf{X}:=\{X_1,...,X_n\}$ we can switch from the model considered in Baraud to ours by setting $s(x)=\left\langle \beta,x\right\rangle$. The seminorm $\|\cdot\|_n$ becomes our $\|\cdot\|_{\Gamma_n}$ and the least square contrast is now :
\[\tilde\gamma_n:\phi\mapsto\frac 1n\sum_{i=1}^n\left(Y_i-\phi(X_i)\right)^2.\]
Our last task consists in identifying the class of models, namely the $S_m$'s. Still sticking to Baraud's notation the latter should be subspaces of $\mathbb{L}^{2}\left(\mathbb{H},\left\Vert \cdot \right\Vert _{n}\right) $. It is simple to see that the following collection suits:
\[\hat{S}_m:=\text{span}\left\{\hat\psi_j, j=1,...,m \right\} \subset \mathbb{L}^{2}\left(\mathbb{H},\left\Vert \cdot \right\Vert _{n}\right), m=1,...,N_n,\]
through the identification mentioned above that is  $\hat\psi_j(x)=\left\langle \hat\psi_j,x\right\rangle$.
%
As Assumption~\textbf{H1} is supposed to be verified, we are now ready to apply Corollary 3.1 of Baraud~\cite{baraud_model_2000} with $q=1$ and obtain that a.s.
\begin{equation}\label{majocond}\E_{\mathbf X}\left[\|\beta-\hat\beta_{\hat m}\|_n^2\right]\leq C(\theta)\inf_{m\in\widehat{\mathcal{M}_n}}\left(\|\beta-\Pi_{\hat S_m}\beta\|_n^2+\pen(m)\right)+\frac{\Theta_p}{n}\sigma^2,
\end{equation}
with
\[\Theta_p:=C'(\theta,p)\frac{\tau_p}{\sigma^p}\left(1+\sum_{m\in\widehat{\Mnr}}m^{-(p/2-2)}\right)\leq C''(\theta,p)\frac{\tau_p}{\sigma^p},\]
and $\E_\mathbf{X}$ denotes the conditional expectation with respect to $\mathbf{X}$.

Noticing that we set earlier $\Pi_{\hat S_m}=\hat\Pi_m$, Equation~(\ref{majocond}) leads to :


\[\E_{\mathbf X}\left[\|\beta-\hat\beta_{\hat m}\|_{\Gamma_n}^2\right]\leq C(\theta)\inf_{m\in\widehat{\mathcal{M}_n}}\left(\|\beta-\hat\Pi_m\beta\|_{\Gamma_n}^2+\pen(m)\right)+\frac{C''(\theta,p)}{n}\sigma^2.\] 
Now, we must ensure that the dimension of the oracle (i.e. the dimension that realises the best bias-variance compromise in $\mathcal M_n$) is included in $\widehat{\Mnr}$. Remark that, if $m>\hat{N}_n$, 
\[\|\beta-\hat\Pi_{\hat N_n}\beta\|_{\Gamma_n}^2+\pen(\hat N_n)\leq \sum_{j>\hat N_n}\hat\lambda_j<\beta,\hat\psi_j>^2+\pen(m)+\frac{(1+\theta)\sigma^2}{n}(\hat{N_n}-m),\]
moreover 
\begin{eqnarray*}
\sum_{j>\hat N_n+1}\hat\lambda_j<\beta,\hat\psi_j>^2&=&\|\beta-\hat\Pi_m\beta\|_{\Gamma_n}^2+\sum_{j=\hat N_n+1}^m\hat\lambda_j<\beta,\hat\psi_j>^2\leq\|\beta-\hat\Pi_m\beta\|_{\Gamma_n}^2+\mathfrak{s}_n\|\beta\|^2,
\end{eqnarray*}
since, for all $j>\hat{N}_n$, $\hat\lambda_j\leq\mathfrak{s}_n$. Therefore 
\begin{equation}\label{hMnr2Mnr}
\|\beta-\hat\Pi_{\hat N_n}\beta\|_{\Gamma_n}^2+\pen(\hat N_n)\leq \|\beta-\hat\Pi_m\beta\|_{\Gamma_n}^2+\pen(m)+\frac{\|\beta\|^2}{n^2},
\end{equation}
and we obtain, for all $m\in\Mnr$
\[\E_{\mathbf X}\left[\|\beta-\hat\beta_{\hat m}\|_{\Gamma_n}^2\right]\leq C(\theta)\left(\|\beta-\hat\Pi_m\beta\|_{\Gamma_n}^2+\pen(m)\right)+\frac{C''(\theta,p)}{n}(\sigma^2+\|\beta\|^2).\]
The proof is completed by taking expectation on both sides of the last inequality. 
\end{proof}
\end{proposition}
\subsubsection{Bound on the empirical risk with unknown noise variance}
\begin{proposition}\label{risknvi}Suppose that Assumption~\textbf{H1} is fulfilled. We have
\begin{equation*}
\mathbb E[\|\tbuv-\beta\|_{\Gamma_n}^2]\leq C\inf_{m\in\Mnr}\left\{\mathbb E[\|\beta-\hat\Pi_m\beta\|_{\Gamma_n}^2]+\pen^{(uv)}(m)\right\}+\frac{C'}{n}(\sigma^2+\|\beta\|^2+\tau_p^{2/p}), \end{equation*}
with $C,C'>0$ depends only on $\theta$, $p$ and $\delta$. 

\end{proposition}
\begin{proof}
Because of the random penalty, we cannot proceed as in the proof of Proposition~\ref{risknvc}. The following proof is based on contrast decomposition and control of the remaining empirical process. More precisely, by definitions of $\hat m^{(uv)}$ and $\hat\beta_m$:
\begin{equation*}\gamma_n(\tbuv)-\gamma_n(\hat\Pi_m\beta)\leq \penh(m)-\penh(\hat m^{(uv)}),\end{equation*}
and
\[\gamma_n(\tbuv)-\gamma_n(\hat\Pi_m\beta)=\|\beta-\tbuv\|_{\Gamma_n}^2-\|\beta-\hat \Pi_m\beta\|_{\Gamma_n}^2+2\nu_n(\hat\Pi_m\beta-\tbuv),\]
with:
\[\nu_n(t):=\frac 1n\sum_{i=1}^n\varepsilon_i<t,X_i>,\]
an empirical linear centred process. 
Then: 
\begin{equation}\label{majocont}\|\beta-\tbuv\|_{\Gamma_n}^2\leq \|\beta-\hat \Pi_m\beta\|_{\Gamma_n}^2+\penh(m)-\penh(\hat m^{(uv)})+2\nu_n(\tbuv-\hat\Pi_m\beta).\end{equation}
The first step is to replace the random function $\penh$ by its empirical counterpart $\pen^{(uv)}$, this can be done by using the results of Lemma~\ref{pen} in the Appendix directly in Equation~(\ref{majocont}):
\begin{eqnarray}\E_\mathbf{X}[\|\beta-\tbuv\|_{\Gamma_n}^2]&\leq&\left(1+\kappa\frac mn\right)\|\beta-\hat \Pi_m\beta\|_{\Gamma_n}^2+\E_\mathbf{X}[\pen^{(uv)}(m)-\pen^{(uv)}(\hat m^{(uv)})]\nonumber\\
&&\hspace{-4.5cm}\label{majocont2}+\E_\mathbf{X}\left[2\left(1+\frac{\kappa\hmuv}{n}\right)\nu_n(\tbuv-\hat\Pi_m\beta)\right]+\frac{\tau_p^{2/p}+\|\beta\|_\Gamma^2+\sigma^2}{n},\end{eqnarray}
where $\kappa:=2(\theta+1)$.

Then the last step consists in controlling the empirical linear process $\nu_n$ on $\hat S_{m\vee \hat m}$. Remark that for all $\delta>0$, for all $m\in\widehat{\Mnr}$,
\begin{equation}\label{majocont3}
2\nu_n(\tbuv-\hat\Pi_m\beta)\leq \frac1\theta \|\tbuv-\hat\Pi_m\beta\|_{\Gamma_n}^2+\theta\sup_{\substack{f\in \hat S_{\hmuv\vee m}\\\|f\|_{\Gamma_n}=1}}\nu_n^2(f),
\end{equation}
since for all $x,y\in \R$ and $\theta>0$, $2xy\leq \theta^{-1} x^2+\theta y^2$.\\
Let $ p(m,m'):=2(1+\delta)\frac{m\vee m'}{n}\sigma^2$, remark that $\pen(m)+\pen(m')\geq p(m,m')$. Then, since $\theta>4$ and $\hat N_n\leq n/\kappa$, gathering equations~(\ref{majocont2}) and~(\ref{majocont3}) we obtain:
\begin{eqnarray*}
\left(1-\frac4\theta\right)\E_\mathbf{X}\left[\|\tbuv-\beta\|_{\Gamma_n}^2\right]&\leq& \left(2+\frac4\theta\right)\|\beta-\hat\Pi_m\beta\|_{\Gamma_n}^2\\
&&\hspace{-3cm}+2\penuv(m)+2\theta\EX\left[\left(\sup_{\substack{f\in \hat S_{\hmuv\vee m}\\\|f\|_{\Gamma_n}=1}}\nu_n^2(f)-p(m,\hmuv)\right)_+\right].
\end{eqnarray*} 
Then the last step is to bound the variations of $\sup_{\substack{f\in \hat S_{\hmuv\vee m}\\\|f\|_{\Gamma_n}=1}}\nu_n^2(f)$ (which can be seen as a variance term) around $p(m,m')$, the results comes from Lemma~\ref{lem:cont_nu_n} detailed below:
\[\E_\mathbf{X}\left[\|\tbuv-\beta\|_{\Gamma_n}^2\right]\leq C(\theta)\min_{m\in\widehat{\Mnr}} \left\{\|\beta-\hat\Pi_m\beta\|_{\Gamma_n}^2+\penuv(m)\right\}+\frac{C(p,\delta)}{n}\sigma^2.\] 
Then we conclude as in the proof of Proposition~\ref{risknvc} by Inequality~(\ref{hMnr2Mnr}).  
\end{proof}
For sake of clarity, Lemma~\ref{lem:cont_nu_n}, which is the key of the previous result, is given below.   
\begin{lemma}\label{lem:cont_nu_n} Suppose that Assumption~\textbf{H1} is fulfilled. Let $p(m,m')=2(1+\delta)\frac{m\vee m'}{n}\sigma^2$,  then for all $m\in\widehat{\Mnr}$,
\[\sum_{m'\in\widehat{\Mnr}}\E_\mathbf{X}\left[\left(\sup_{\substack{f\in \hat S_{m\vee m'}\\\|f\|_{\Gamma_n}=1}}\nu_n^2(f)-p(m,m')\right)_+\right]\leq \frac{C(p,\delta)}{n}\sigma^2.\]
\end{lemma}
The proof of this lemma is given in Section~\ref{subsec:control_nun} and relies on results of Baraud~\cite{baraud_model_2000} based on Talagrand's Inequality. 
\subsection{Oracle inequality}
\label{subsec:oracle}

In this section, we derive an oracle-inequality for the risk associated to the prediction error. We define first an allipsoid of $\Ld$
\[\mathcal{W}_r^R:= \left\{f\in\Ld,\ \sum_{j\geq 1}j^r<f,\psi_j>^2\leq R^2\right\}.\]

\begin{theorem}\label{majo}
Suppose that assumptions~\textbf{H1},~\textbf{H2}, \textbf{H3} and \textbf{H4} hold and that the decreasing rate of $(\lambda_j)_{j \geq 1}$ is given by \textbf{(P)} or \textbf{(E)}. Then, for all slope function $\beta\in\Ld$, if $n\geq 6$:
\begin{eqnarray}\label{oracle}\E[\|\widetilde\beta-\beta\|_\Gamma^2]&\leq &C_1\left(\min_{m\in\Mnr}\left(\E[\|\beta-\hat\Pi_m\beta\|_\Gamma^2]+\E[\|\beta-\hat\Pi_m\beta\|_{\Gamma_n}^2]+\pen(m)\right)\right)+\frac{C_2}{n}\left(1+\|\beta\|^2\right)\end{eqnarray}
where $C_1>0$ and $C_2>0$ are independent of $\beta$ and $n$. 

If, in addition, $\beta\in\mathcal W_r^R$ --- with the condition $a+r>2$ in the polynomial case~\textbf{(P)} --- we have 
\begin{equation}\label{oracle2}\E[\|\widetilde\beta-\beta\|_\Gamma^2]\leq C_1'\left(\min_{m\in\Mnr}\left(\E[\|\beta-\hat\Pi_m\beta\|_\Gamma^2]+\pen(m)\right)\right)+\frac{C_2'}{n}\left(1+\|\beta\|^2\right),\end{equation}
where the constants $C_1'>0$ and $C_2'>0$ do not depend on $\beta$ or $n$. 
\end{theorem}

\hspace{1cm}
\begin{remark}The condition $a+r>2$ is verified as soon as $a\geq 2$ without condition on the regularity parameter $r$ of the slope $\beta$. Note that if $X$ is a Brownian motion, the sequence $(\lambda_j)_{j\geq 1}$ associated to the process $X$ verifies \textbf{(P)} with $a=2$. Then we do not need additional condition on $r$ if $X$ is smoother than the Brownian motion.    
\end{remark}
\paragraph{Sketch of proof.}
The core of the proof relies on the bounds on the empirical risk given in Proposition~\ref{risknvc}, for the known variance case, and Proposition~\ref{risknvi} for the unknown variance case.
Then it remains to replace the empirical risk appearing in propositions~\ref{risknvc} and \ref{risknvi} by the risk associated to the prediction error in order to obtain the final oracle-inequality. This is done with the results of Lemma~\ref{deltanb} which allows to control the set 
\begin{equation}\label{deltan}
 \Delta_n:=\lbrace\forall f\in\mathcal{\hat S}_n, \|f\|_\Gamma^2\leq \rho_0\|f\|_{\Gamma_n}^2\rbrace, 
 \end{equation}
 where $\rho_0>1$ is a constant and we define $\hat{\mathcal S}_n:=\hat S_{\hat N_n}$.

\begin{proof}

The following equality holds: 
 \begin{equation*}
 \E[\|\widetilde{\beta}-\beta\|_\Gamma^2]= \E[\|\widetilde{\beta}-\beta\|_\Gamma^2\mathbf{1}_{\Delta_n}]+\E[\|\widetilde{\beta}-\beta\|_\Gamma^2\mathbf{1}_{\Delta_n^\complement}],\end{equation*}
 where, for a set $A$, we denote by $A^\complement$ its complement.

Lemma~\ref{T1} in Section~\ref{proofs} allows to bound the second term of this inequality. Thus the end of the proof will be devoted to upper-bound the first term.

 We remark that, for all $m\in\Mnr$,
\begin{eqnarray}\|\hat{\beta}_{\hat{m}}-\beta\|_\Gamma\mathbf{1}_{\Delta_n}&\leq& \|\hat\beta_{\hat m}-\hat\Pi_m\beta\|_\Gamma\mathbf 1_{\Delta_n}+\|\beta-\hat\Pi_m\beta\|_\Gamma\nonumber\\
\label{finalvc}&\leq&\sqrt\rho_0\|\hat\beta_{\hat m}-\beta\|_{\Gamma_n}+\sqrt\rho_0\|\beta-\hat\Pi_m\beta\|_{\Gamma_n}+ \|\beta-\hat\Pi_m\beta\|_\Gamma .\end{eqnarray}

By propositions~\ref{risknvc} and~\ref{risknvi}, for all $m\in\Mnr$:
\[\mathbb E[\|\hat\beta_{\hat m}-\beta\|_{\Gamma_n}^2]\leq C(p,\theta,\delta,\sigma^2,\tau_p)\left(\mathbb E[\|\beta-\hat\Pi_m\beta\|_{\Gamma_n}^2]+\pen(m)+\frac{1+\|\beta\|^2}{n}\right),\]
 and by Equation~(\ref{finalvc})
\begin{eqnarray*}
\E[\|\hat{\beta}_{\hat{m}}-\beta\|_\Gamma^2\mathbf{1}_{\Delta_n}]&\hspace{-0.3cm}\leq&\hspace{-0.3cm}C(p,\theta,\delta,\sigma^2,\tau_p,\rho_0)\left(\E[\|\beta-\hat\Pi_m\beta\|_{\Gamma_n}^2]+\E[\|\beta-\hat\Pi_m\beta\|_\Gamma^2]+\pen(m)+\frac{1+\|\beta\|^2}{n}\right),
\end{eqnarray*}
and Equation~(\ref{oracle}) follows.

Then Equation~(\ref{oracle2}) comes from Equation~(\ref{oracle}) and Lemma~\ref{control_biais}.
\end{proof}

\subsection{Convergence rates}
\label{subsec:convergence_rates}

As a direct consequence of the oracle-inequality given in Theorem~\ref{majo}, associated with the control of the random projector on the spaces $\hat S_m$ given in Lemma~\ref{normnGamma}, we derive uniform bounds on the risk of our estimators on the ellipsoids $\mathcal{W}_r^R$.


\begin{theorem}\label{majo_risk}Assume that the assumptions of Theorem~\ref{majo} are fulfilled. For all $r>0$ and $R>0$:
\begin{description}
\item[Polynomial case. ] If \textbf{(P)} holds with $a+r>2$ then:
\begin{equation}\label{vit_P}
\sup_{\beta\in \mathcal W_r^R}\E[\|\widetilde{\beta}-\beta\|_\Gamma^2]\leq C_Pn^{-(a+r)/(a+r+1)};
\end{equation}
\item[Exponential case. ]  If \textbf{(E)} holds then:
\begin{equation}\label{vit_E}
\sup_{\beta\in\mathcal W^R_r}\E[\|\widetilde{\beta}-\beta\|_\Gamma^2]\leq C_En^{-1}(\ln n)^{1/a},
\end{equation}
with $C_P$ and $C_E$ independent of $n$.
\end{description}

\hspace{1cm}
\begin{remark} In the case where the noise $\varepsilon$ is Gaussian, the bounds~(\ref{vit_P}) and (\ref{vit_E}) coincide with the minimal bounds given by Cardot and Johannes~\cite{cardot_thresholding_2010}.
\end{remark}

\begin{proof} 
Let us start with the polynomial case \textbf{(P)}.
By Theorem~\ref{majo}, we have: 
\[\E[\|\widetilde\beta-\beta\|_\Gamma^2]\leq C\left(\min_{m\in\Mnr}\left(\E[\|\beta-\hat\Pi_m\beta\|_\Gamma^2]+\pen(m)\right)+\frac1n(1+\|\beta\|^2)\right), \]
with $C$ independent of $\beta$ and $n$.
Denote by $\hat\Pi_m$ the orthogonal projector onto $S_m=\text{Span}\{\psi_1,\hdots,\psi_m\}$, by Lemma~\ref{normnGamma} 
\begin{eqnarray*}\|\beta-\hat\Pi_m\beta\|_\Gamma^2&\leq& 2\|\beta-\Pi_m\beta\|_\Gamma^2+C_1\frac{\ln^3 m}{n}m^{\max\{(1-r)_+,2-a-r\}}+C_2\frac{\ln^5 m\ln^4 n}{n^2}m^{\max\{(2-a+(7-r)_+)_+,2-a+(5-r)_+\}}, 
\end{eqnarray*}
with $C_1, C_2>0$, independent of $\beta$ and $n$. 
Now since $\beta\in\mathcal W_r^R$,
\[\E[\|\beta-\Pi_m\beta\|_\Gamma^2]=\sum_{j\geq 1}\lambda_j<\beta,\psi_j>^2\leq m^{-a-r}\sum_{j\geq m}j^r<\beta,\psi_j>^2\leq R m^{-a-r}.\]

We can see easily that it is possible to define a sequence of integers $(m_n^*)_{n\in\N^*}$ such that
\[m_n^*\asymp n^{1/a+r+1}\text{ and }m_n^*\leq N_n\text{ for all }n\in\N^*,\]
where for two sequences $(a_k)_{k\geq 1}$ and $(b_k)_{k\geq 1}$, we note $a_k\lesssim b_k$ if there exists some constant $c>0$ such that, for all $k\leq 1$, $a_k\leq c b_k$ and we note also $a_k\asymp b_k$ if $a_k\lesssim b_k $ and $b_k\lesssim a_k$.
Now considerations above lead us to
\[\sup_{\beta\in\mathcal W_r^R}\E[\|\beta-\hat\Pi_{m_n^*}\beta\|_\Gamma^2]\lesssim n^{-\frac{a+r}{a+r+1}}, \]
as soon as $a+r>2$ and in addition
\[\pen(m_n^*)\lesssim n^{-\frac{a+r}{a+r+1}}, \]
which leads to the expected bound.

The exponential case \textbf{(E)} is treated similarly with $m_n^*\asymp \ln^{1/a} n$.
\end{proof}
\end{theorem}

\section{Numerical results}
\label{simulation}
\subsection{Simulation method}
Following the method proposed by Hall and Hosseini-Nasab~\cite{hall_properties_2006}, we simulate the random function $X$ in the following way
\begin{equation}\label{simu_X}
	X=\sum_{j=1}^J\sqrt{\lambda_j}\xi_j\psi_j,
\end{equation}
where, for all $j\geq 1$, $\psi_j(x)=\sqrt{2}\sin(\pi(j-0.5)x)$ and $\{\xi_j,\ j=1,...,J\}$ is independent and follows the standard normal distribution. This sequence of functions $(\psi_j)_{j\geq 1}$ has been chosen so that if $J$ is sufficiently high and if $\lambda_j=(j-0.5)^{-2}\pi^{-2}$, we obtain a Brownian motion (see~Ash and Gardner~\cite{ash_topics_1975}). In order to see how the decreasing rate of $(\lambda_j)_{j\geq 1}$ influences the estimation, we take three different sequences: 
\[\lambda_j^{P1}=j^{-2},\ \lambda_j^{P2}=j^{-3}\text{ and }\lambda_j^{E}=e^{-j}.\]
It is interesting to note from Equation~(\ref{simu_X}) that the higher the rate of decrease of the $\lambda_j$'s, the better the regularity of the function $X$.

The function $X$ is then discretized over $p=100$ equispaced points \\$\left\{t_j=\frac{j-1}{p}, j=1,...,p\right\}$. We take $\varepsilon\sim\mathcal{N}(0,\sigma^2)$ and $\sigma^2=0.01$. We consider here two different slope functions
\[\beta_1(x)=\ln(15x^2+10)+\cos(4\pi x)\text{ (see \cite{cardot_functional_1999}) and }\beta_2(x)=e^{(x-0.3)^2/0.05}\cos(4\pi x).\]

\subsection{Comparison with cross validation}
We compare our dimension selection criterion with two cross validation criteria frequently used in practice. The first method consists in minimizing 
\[GCV(m):=\frac{\sum_{i=1}^n(Y_i-\hat Y_i)}{\left(1-\text{tr}(\mathbf{H}_m)/n\right)^2},\]
where $ \hat Y_i:=\int_0^1\hat\beta_m(t) X_i(t)dt$ and $\mathbf H_m$ is the classical Hat matrix defined by $\hat{\mathbf Y}=(\hat Y_1,..., \hat Y_n)'=\mathbf H_m \mathbf Y$. This criterion has been proposed in a similar context by Marx and Eilers~\cite{marx_flexible_1996} and in the context of functional linear models by Cardot~\textit{et al.}~\cite{cardot_spline_2003}. The second one consists in minimizing the criterion
\[CV(m):=\frac 1n \sum_{i=1}^n\left(Y_i-\hat Y_i^{(-i)}\right)^2,\]
which has been proposed in the framework of functional linear model by~Hall and Hosseini-Nasab~\cite{hall_properties_2006}. Here $\hat Y_i^{(-i)}$ is  the value of $Y_i$ predicted from the sample $\{(X_j,Y_j), j\neq i\}$. Note that an immediate drawback of this criterion is that it requires a much longer CPU time than the GCV criterion or our penalized criterion. 

\subsection{Results}

\begin{figure}
\begin{tabular}{ccc}
$\lambda_j^{P1}$&$\lambda_j^{P2}$& $\lambda_j^E$\\
\includegraphics[width=0.3\textwidth]{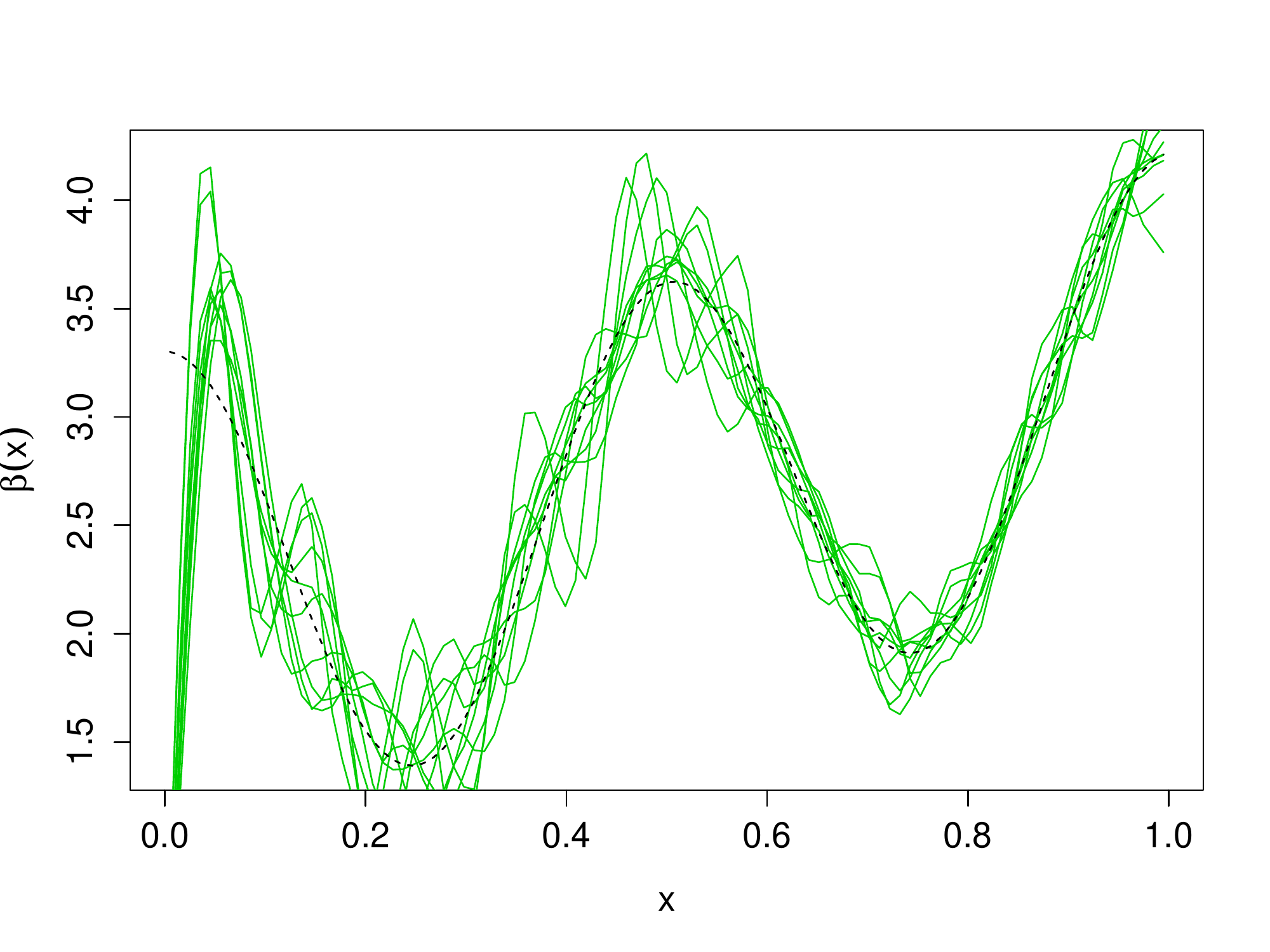}&\includegraphics[width=0.3\textwidth]{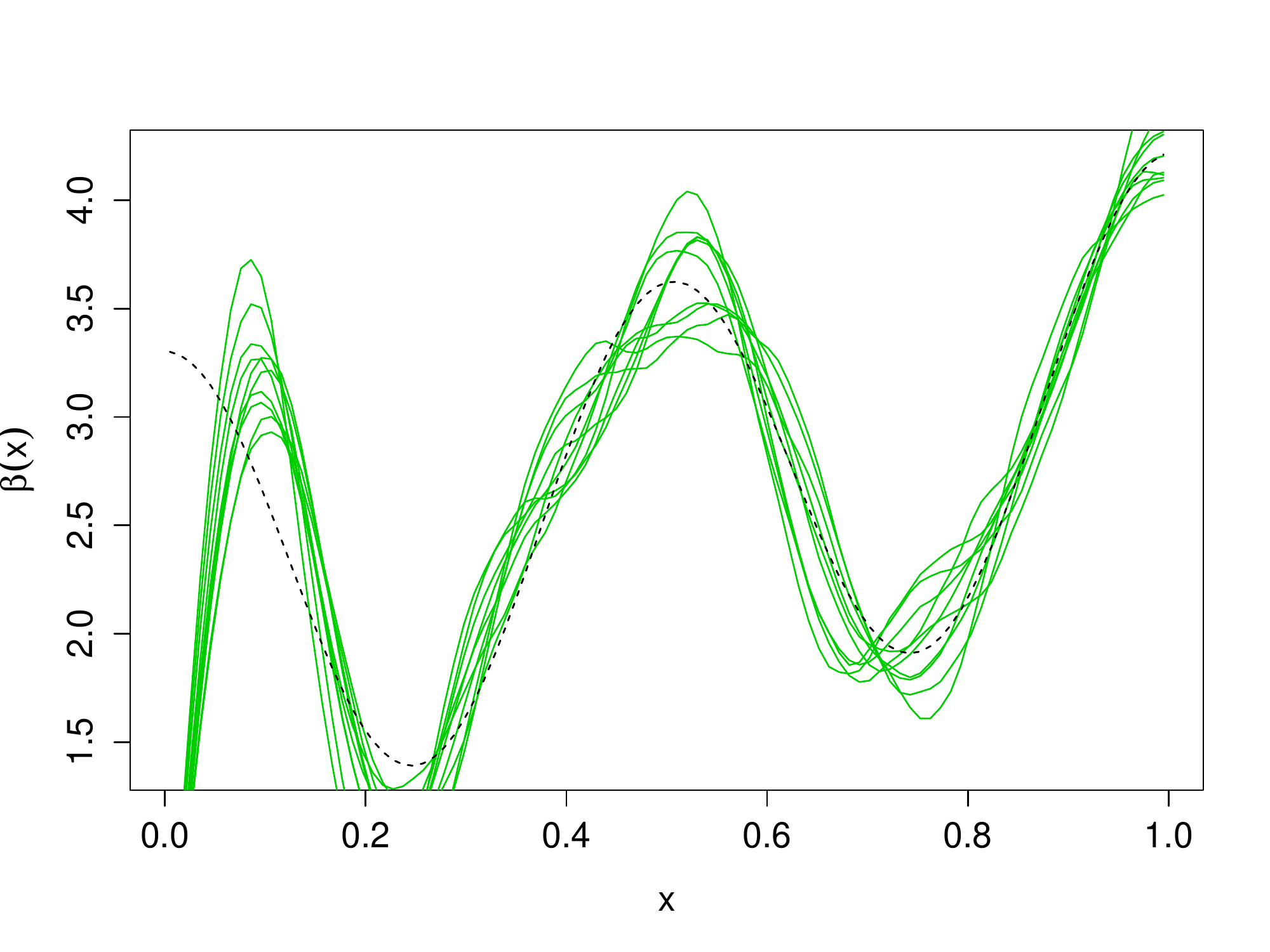}&\includegraphics[width=0.3\textwidth]{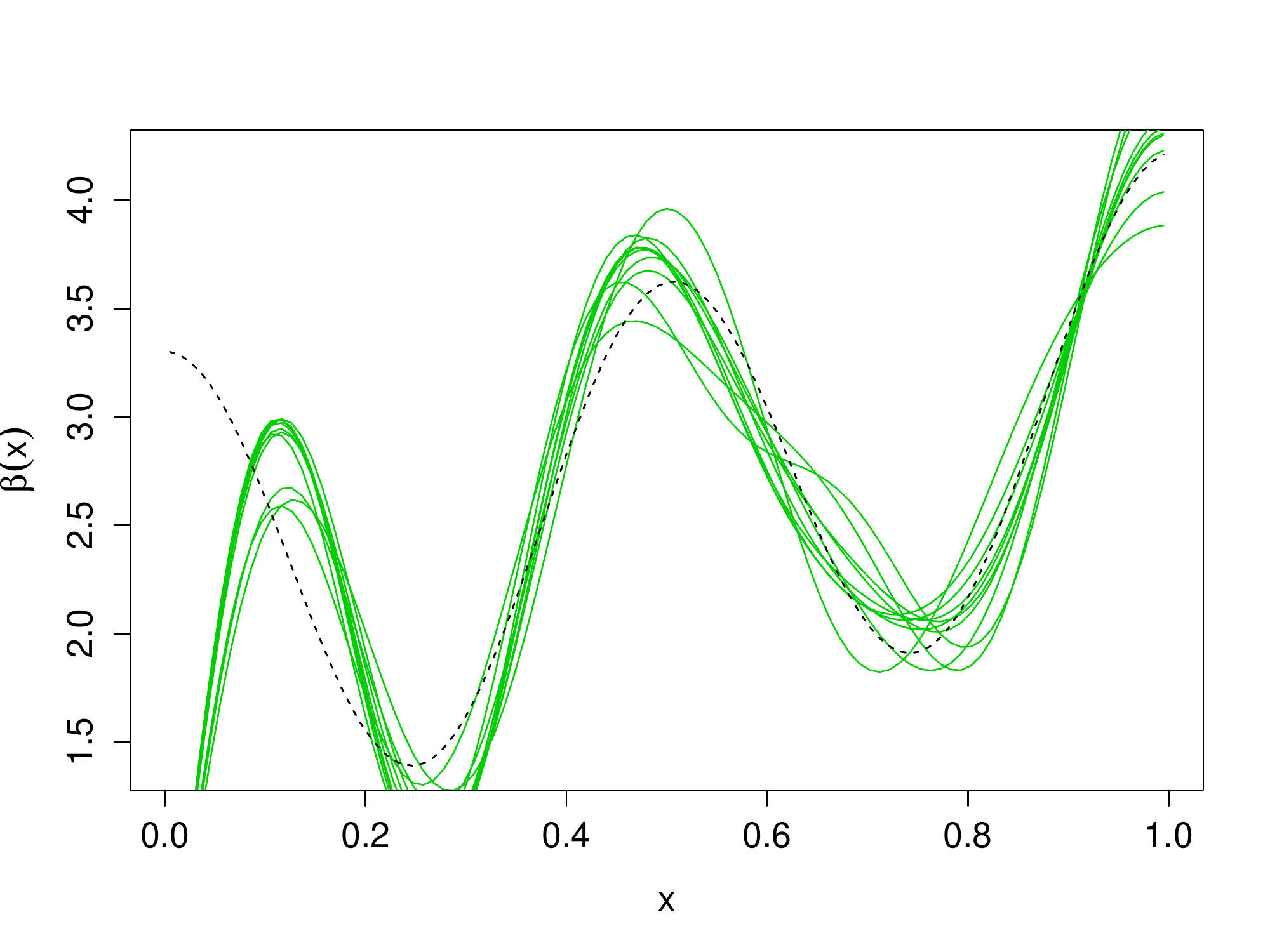}
\end{tabular} 
\caption{\label{estim_beta1}Plot of $\beta_1$ (bold, dashed) and $\tbkv_1$ computed for 10 independent samples of size $n=1000$.}
\end{figure}

\begin{figure}
\begin{tabular}{ccc}
$\lambda_j^{P1}$&$\lambda_j^{P2}$& $\lambda_j^E$\\
\includegraphics[width=0.3\textwidth]{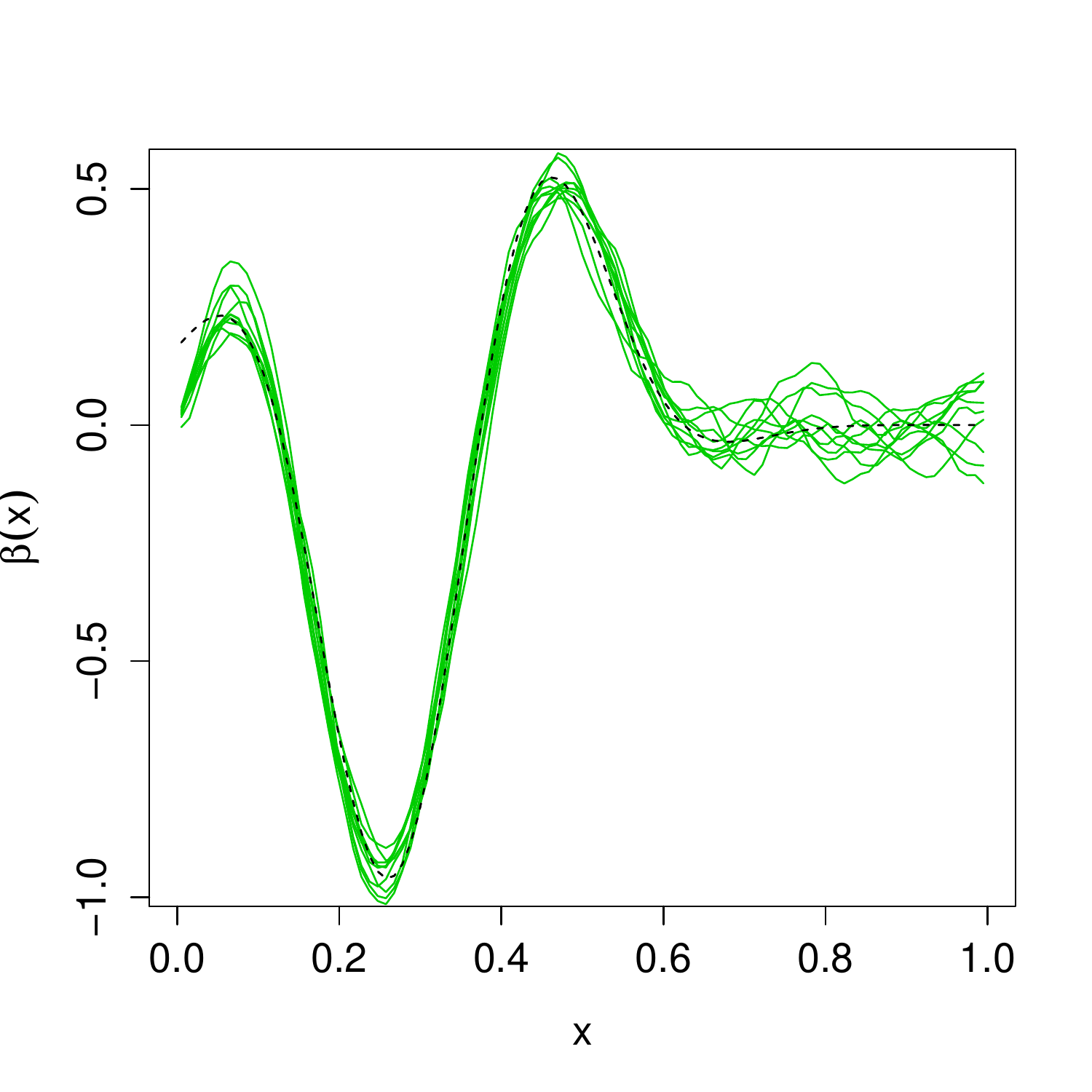}&\includegraphics[width=0.3\textwidth]{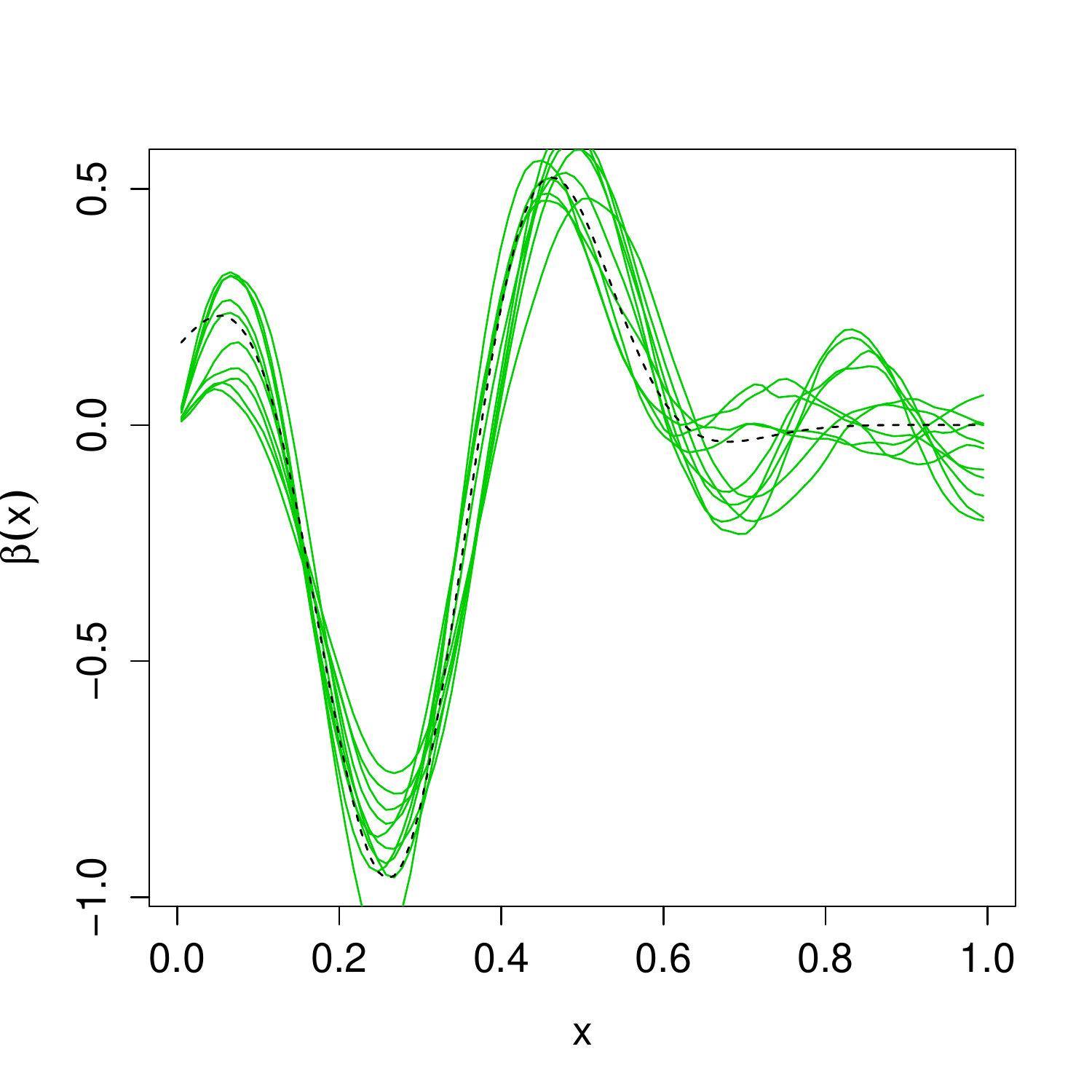}&\includegraphics[width=0.3\textwidth]{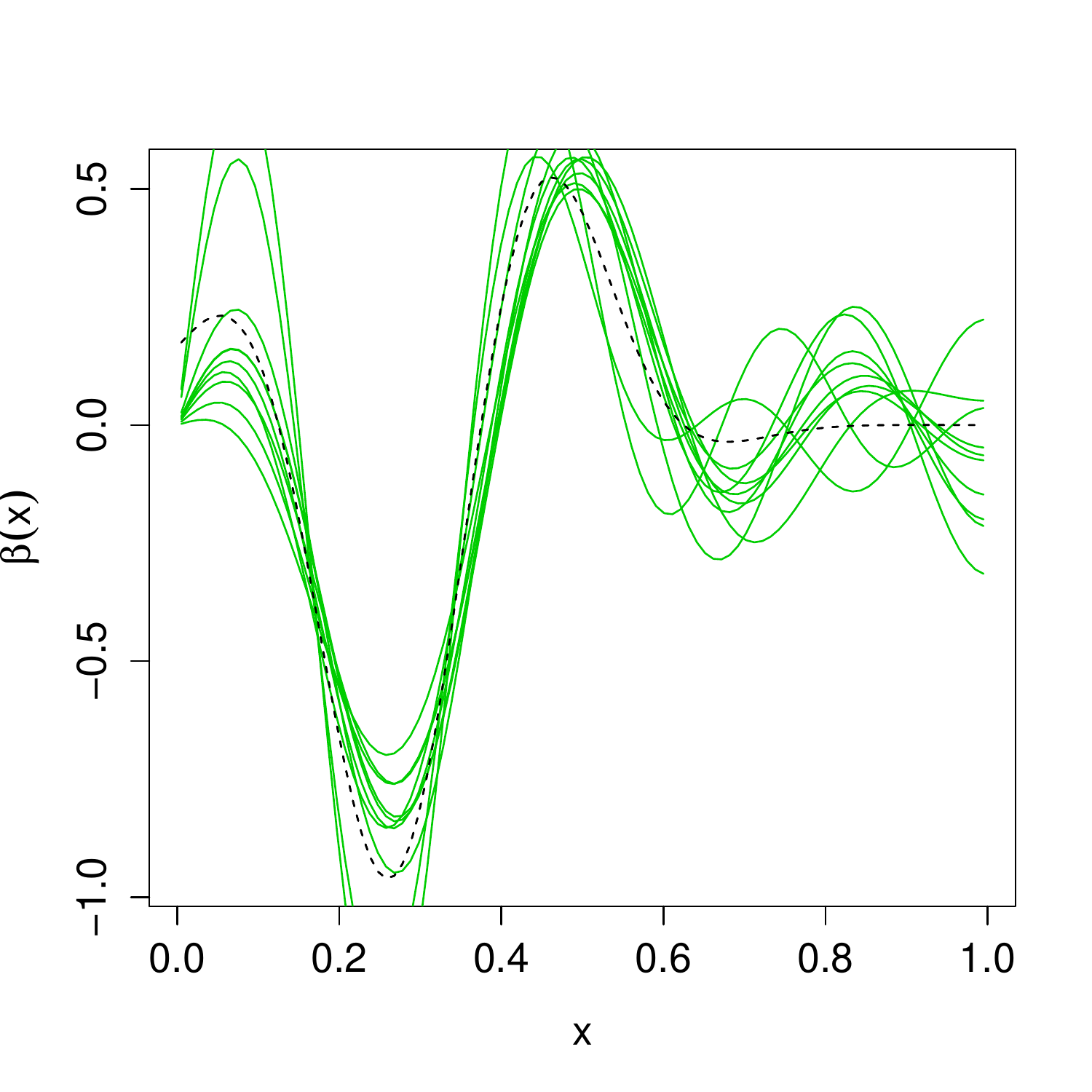}
\end{tabular} 
\caption{\label{estim_beta2}Plot of $\beta_2$ (bold, dashed) and $\tbkv_2$ computed for 10 independent samples of size $n=1000$.}
\end{figure}

As we can see in figures~\ref{estim_beta1} and \ref{estim_beta2}, the regularity of estimators increases when the rate of convergence of the $\lambda_j$'s decreases. This is a specificity of functional PCA: the estimated slope function $\tilde\beta$ is an element of $\text{Im}(\Gamma_n)=\text{span}\{X_1,...,X_n\}$ and thus has the same regularity as the function $X$. It also explains the side effect observed in figures~\ref{estim_beta1} and \ref{estim_beta2} since $X_i(0)=0$ implies that $\tilde\beta(0)=0$.

\begin{table}
\begin{tabular}{|c|c|ccc|ccc|}
\hline
$\lambda_j$&&\multicolumn{3}{c|}{$\beta_1$}&\multicolumn{3}{c|}{$\beta_2$}\\
&&$n=200$&$n=500$&$n=1000$&$n=200$&$n=500$&$n=1000$\\
\hline
$j^{-2}$&kv&\textbf{12.8 $\pm$0.4}&5.9 $\pm$0.2&3.57 $\pm$0.08&\textbf{4.7 $\pm$0.3}&\textbf{1.88 $\pm$0.09}&\textbf{0.89 $\pm$0.05}\\
&uv&\textbf{12.5 $\pm$0.4}&\textbf{5.8 $\pm$0.2}&3.51 $\pm$0.08&\textbf{4.7 $\pm$0.3}&\textbf{1.89 $\pm$0.09}&\textbf{0.89 $\pm$0.05}\\
&GCV&80 $\pm$2&55 $\pm$2&47 $\pm$2&80 $ \pm$2&55 $\pm$2&47 $\pm$2 \\
&CV&\textbf{12.2 $\pm$0.5}&\textbf{5.6 $\pm$0.2}&\textbf{3.34 $\pm$0.09}&5.7 $\pm$0.4&2.2 $\pm$0.2&1.08 $\pm$0.06 \\
\hline
$j^{-3}$&kv&\textbf{6.7 $\pm$0.3}&\textbf{3.3 $\pm$0.1}&\textbf{1.84 $\pm$0.06}&\textbf{5.5 $\pm$0.2}&\textbf{1.8 $\pm$0.1}&\textbf{0.88 $\pm$0.04}\\
&uv&\textbf{6.6 $\pm$0.3}&\textbf{3.2 $\pm$0.1}&\textbf{1.83 $\pm$0.06} &\textbf{5.4 $\pm$0.3}&\textbf{1.8 $\pm$0.1}&\textbf{0.88 $\pm$0.04}\\
&GCV&18.4 $\pm$0.5&12.6 $\pm$0.3&9.3 $\pm$0.2&18.5 $\pm$0.5&12.7 $\pm$0.3&9.5 $\pm$0.2 \\
&CV&7.3 $\pm$0.4&\textbf{3.3 $\pm$0.2}&\textbf{1.78 $\pm$0.07}&\textbf{5.1 $\pm$0.4}&\textbf{2.0 $\pm$0.2}&1.05 $\pm$0.08 \\
\hline
$e^{-j}$&kv&\textbf{4.8 $\pm$0.2}&\textbf{2.05 $\pm$0.08}&\textbf{1.12 $\pm$0.05}&\textbf{5.0 $\pm$0.2}&\textbf{1.9 $\pm$0.1}&\textbf{0.78 $\pm$0.05}\\
&uv&\textbf{4.7 $\pm$0.2}&\textbf{2.03 $\pm$0.08}&\textbf{1.11 $\pm$0.05}&\textbf{4.9 $\pm$0.2}&\textbf{1.82 $\pm$0.09}&\textbf{0.77 $\pm$0.05}\\
&GCV&5.8 $\pm$0.3&2.67 $\pm$0.09&1.45 $\pm$0.05&6.0 $\pm$0.3&2.7 $\pm$0.1&1.40 $\pm$0.05 \\
&CV&\textbf{4.8 $\pm$0.3}&\textbf{2.05 $\pm$0.09}&\textbf{1.10 $\pm$0.05}&\textbf{4.6 $\pm$0.3}&\textbf{1.8 $\pm$0.1}& 0.87 $\pm$0.05\\
\hline
\end{tabular}
\caption{{\label{table}}Mean prediction error ($\times 10^{-4}$) and approximated 95$\%$ confidence interval (calculated from 500 independent samples of size $n=1000$). kv: dimension selected by minimization of~(\ref{critvc}), uv: dimension selected by~(\ref{critvi}).}
\end{table}

The results of Table~\ref{table} indicate that the substitution of the term $\sigma^2$ by the estimator $\hat\sigma_m^2$ in case of unknown variance does not have a significant effect on the quality of estimation. In fact the Monte Carlo study also revealed that the dimension selected by minimization of~(\ref{critvc}) and~(\ref{critvi}) is the same in 70\% to 99\% of cases (percentage depending on the sample size $n$, the decreasing rate of the $\lambda_j$'s and the function $\beta$).

\begin{figure}
\begin{center}
\begin{tabular}{cc}
Estimation of $\beta_1$&\\
\includegraphics[width=0.45\textwidth]{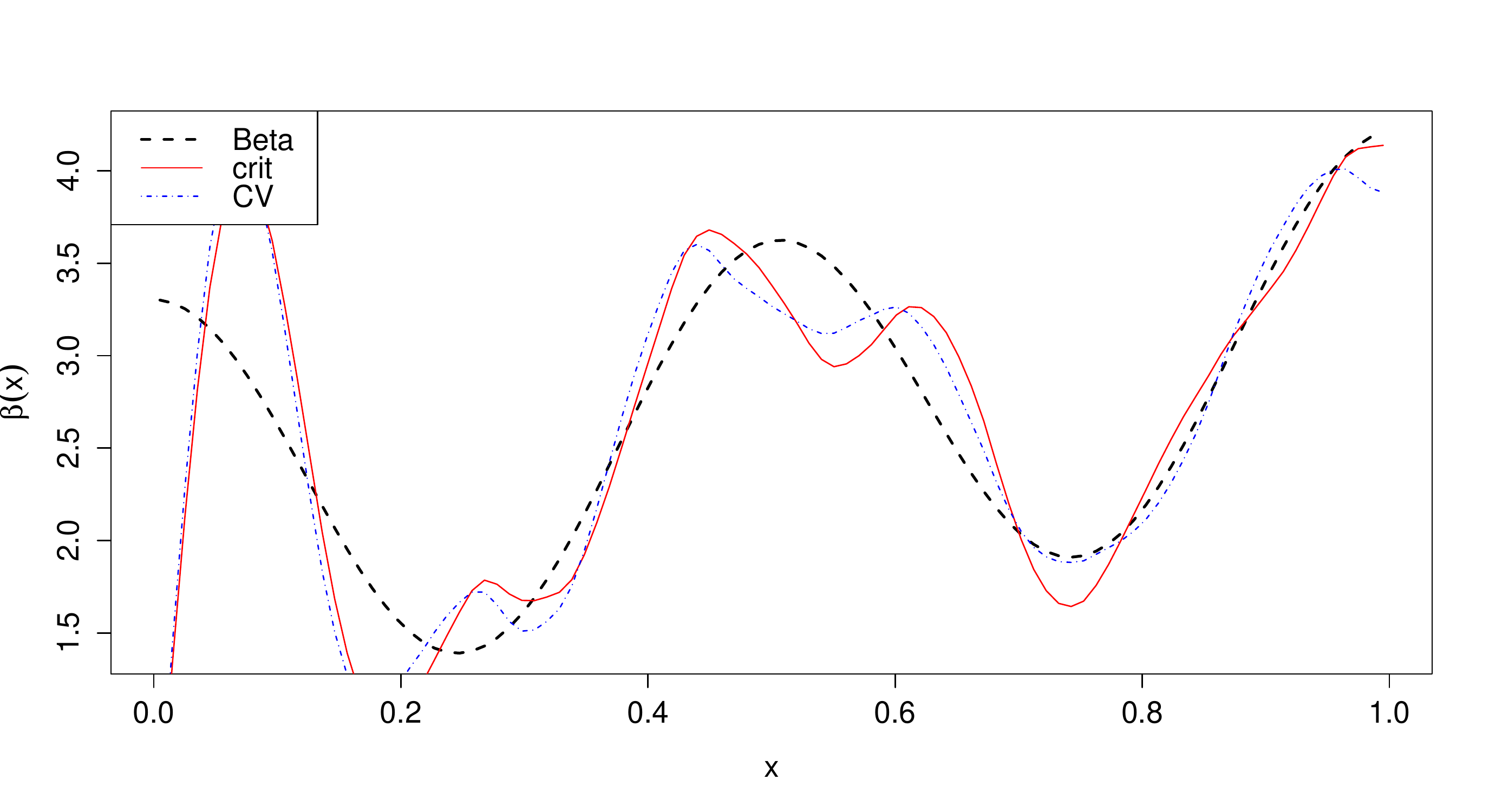}&\includegraphics[width=0.45\textwidth]{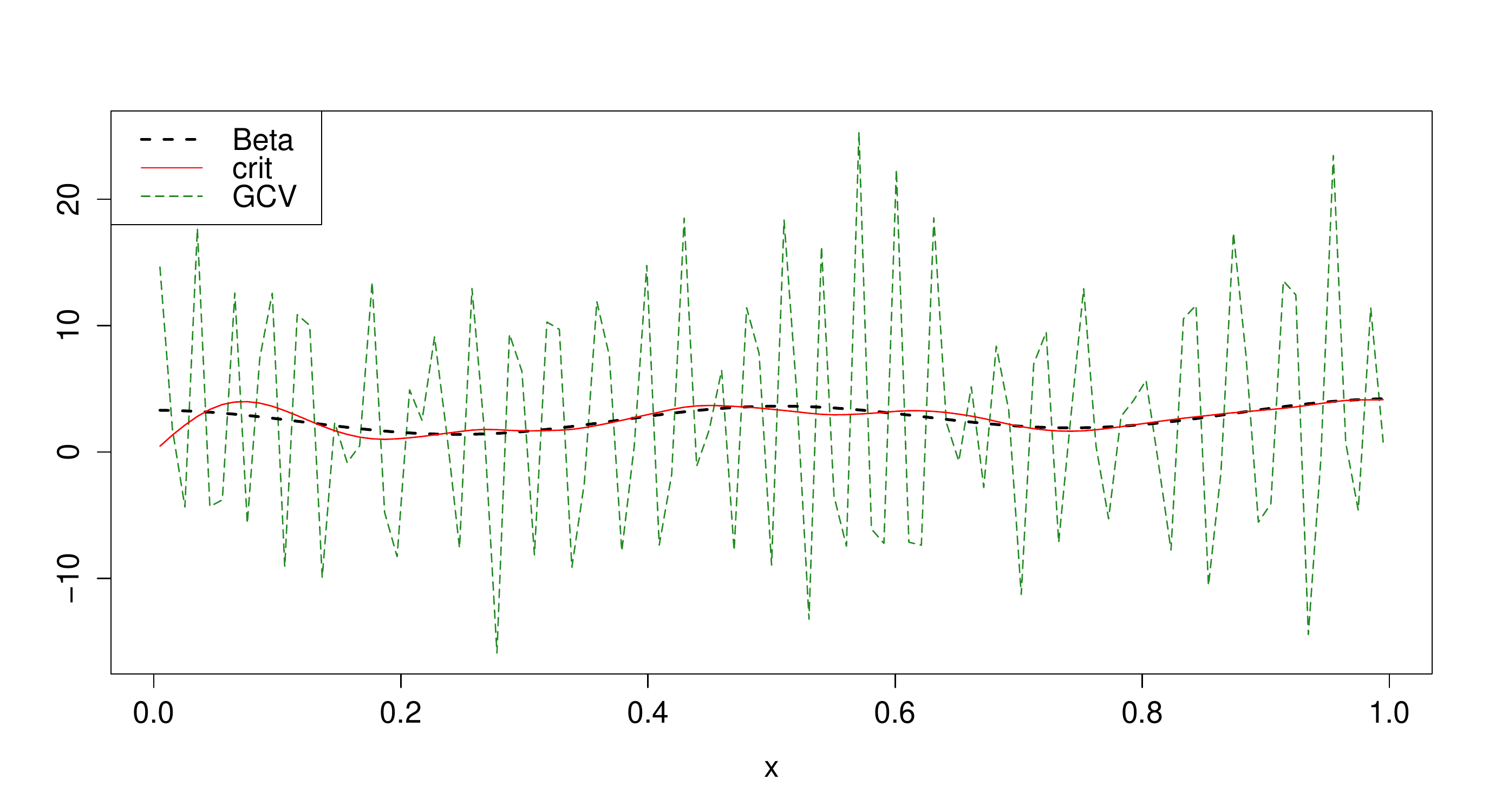}\\
Estimation of $\beta_2$&\\
\includegraphics[width=0.45\textwidth]{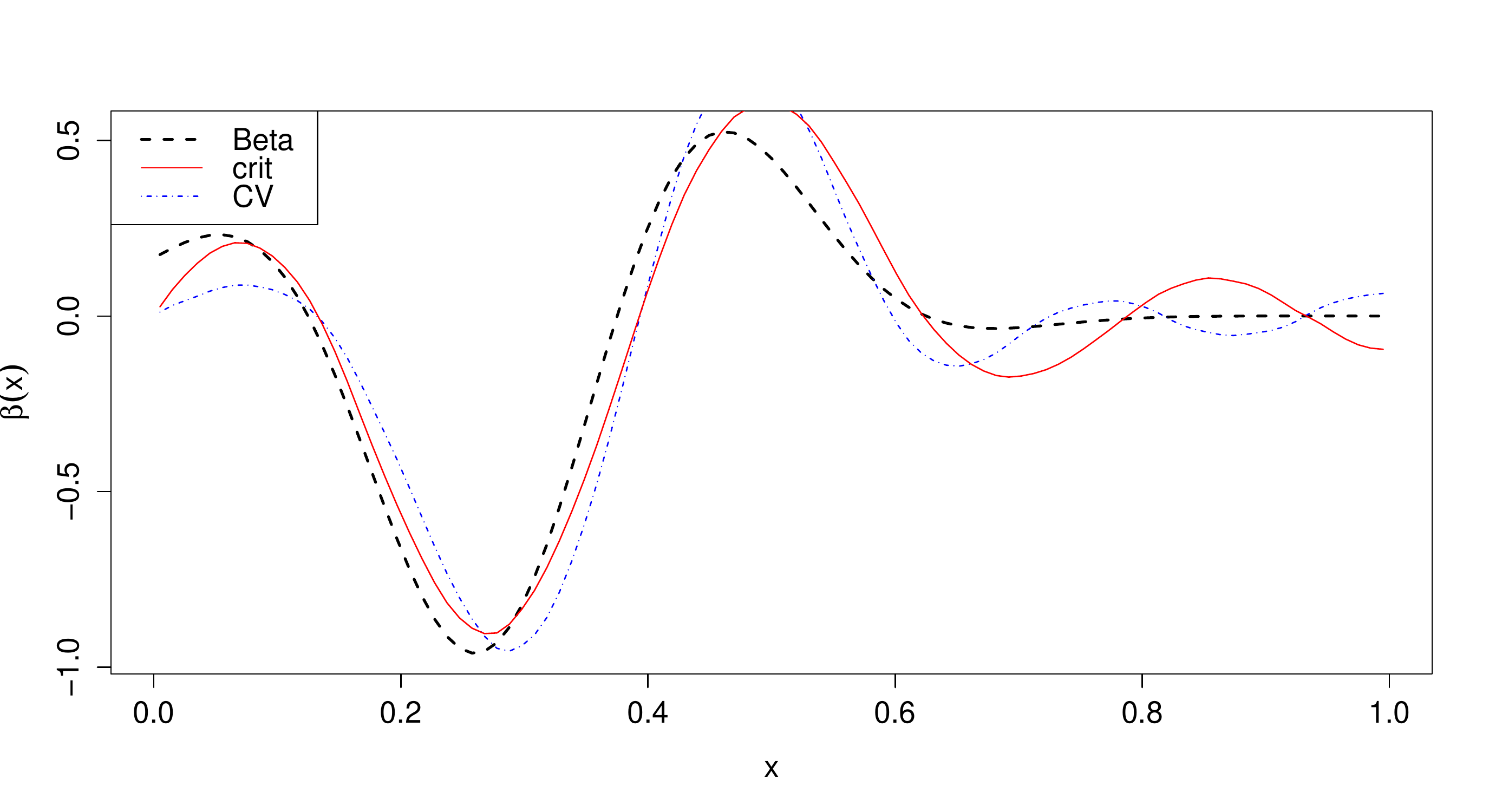}&\includegraphics[width=0.45\textwidth]{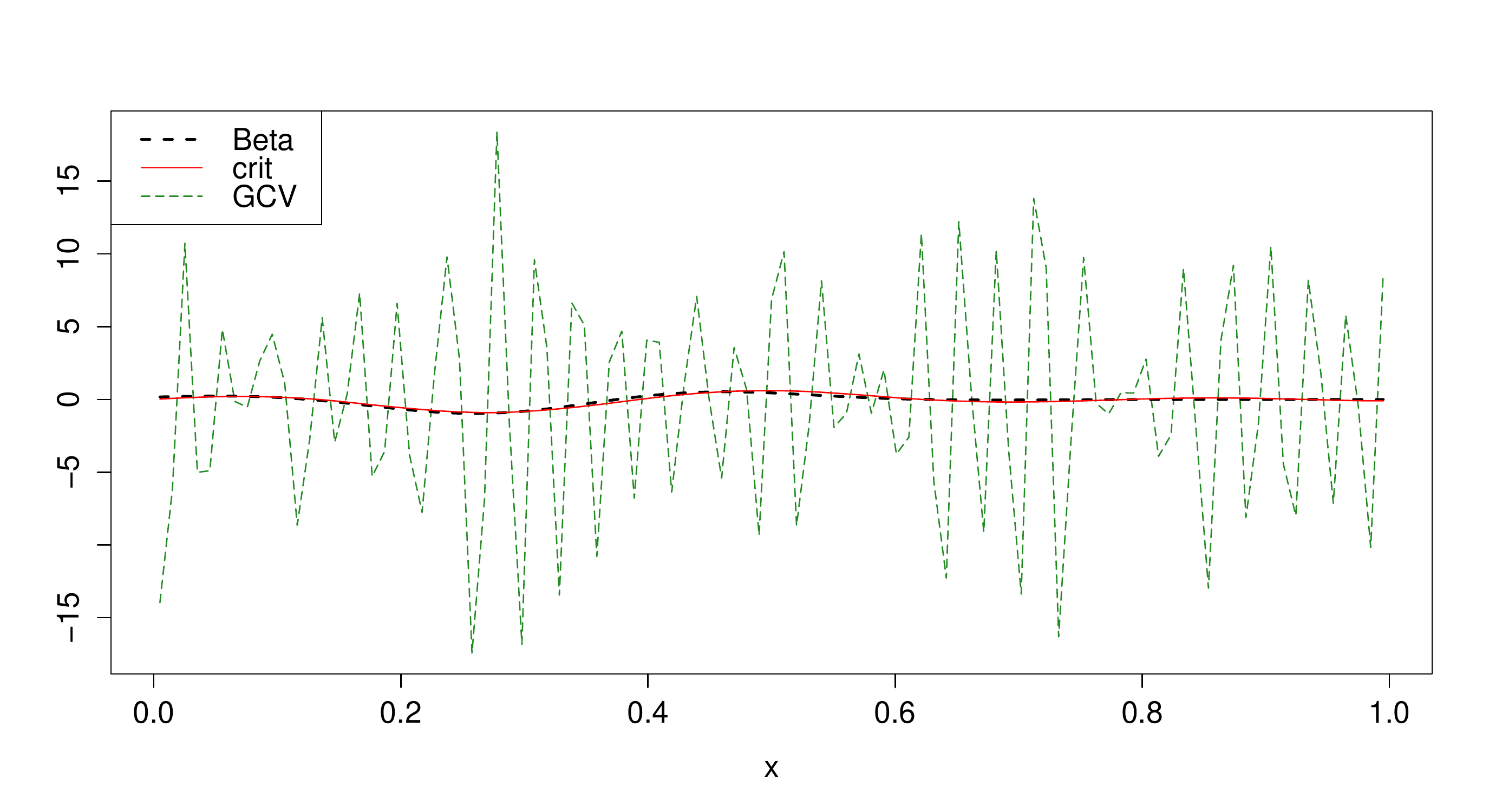}
\end{tabular}
\end{center}

\caption{\label{fig_comp_VC}Left: comparison of estimators $\hat\beta_m$ when $m$ is selected by minimization of the penalized criterion crit defined by (\ref{critvc}) or the CV criterion. Right: comparison with the GCV criterion. $n=2000$, $\lambda_j=j^{-3}$.}
\end{figure}

Moreover, according to Figure~\ref{fig_comp_VC} and Table~\ref{table}, performances of our estimators seem to be quite similar to the functional PCR estimator with dimension selected by minimization of the CV criterion. Conversely, the GCV criterion  selects systematically the highest dimensional model which leads to poor performances. 

\section{Proofs}
\label{proofs}
\subsection{Proof of Lemma~\ref{lem:cont_nu_n} }
\label{subsec:control_nun}

\begin{proof}[Proof of Lemma~\ref{lem:cont_nu_n}]
First denote by $\bar f_\mathbf{X}:=\left(<f,X_1>,...,<f,X_n>\right)'$ and $\bar{\varepsilon}:=\left(\varepsilon_1,...,\varepsilon_n\right)'$. Remark that $\nu_n(f)=\bar f_\mathbf{X}'\bar\varepsilon/n$ and that, by usual properties of orthogonal projectors
\[\sup_{\substack{f\in \hat S_{m}\\\|f\|_{\Gamma_n}=1}}\nu_n(f)=\sup_{\substack{\alpha\in \hat s_{m}\\\alpha'\alpha=n}}\frac{\alpha'\bar\varepsilon}{n}=\frac{1}{\sqrt n}\sup_{\substack{\alpha\in \hat s_{m}\\\alpha'\alpha=1}}\alpha'\bar\varepsilon=\frac{1}{\sqrt n}(\Pi_{\hat s_m}\left(\bar\varepsilon)'\Pi_{\hat s_m}\bar\varepsilon\right)^{1/2}=\frac{1}{\sqrt n}\left(\bar\varepsilon'\Pi_{\hat s_m}\bar\varepsilon\right)^{1/2},\]
where $\hat s_m$ denotes the subspace of $\R^n$ defined by 
\[\hat s_m:=\left\{\alpha\in\R^n,\ \exists f\in\hat S_m,\ \alpha=\bar f_\mathbf{X}\right\}\]
and $\Pi_{\hat s_m}$ is the orthogonal projector onto $\hat s_m$.

Then, by Assumption~\textbf{H1}, applying Corollary 5.1 of Baraud~(2000)~\cite{baraud_model_2000} with $\tilde{A}=\Pi_{\hat s_m}$ we obtain, for all $x>0$,
\[\Pb_\mathbf{X}\left( n\sup_{\substack{f\in \hat S_{m}\\\|f\|_{\Gamma_n}=1}}\nu_n^2(f)\geq m\sigma^2+2\sigma^2\sqrt{mx}+\sigma^2x\right)\leq C(p)\sigma^p\tau_p\frac{m}{x^{p/2}},\]
where $\Pb_\mathbf{X}$ stands for the probability given $\mathbf{X}$.
Then for all $\delta>0$ remark that $\sqrt{mx}\leq \delta m+\delta^{-1}x$ we obtain
\[\Pb_\mathbf{X}\left(\sup_{\substack{f\in \hat S_{m}\\\|f\|_{\Gamma_n}=1}}\nu_n^2(f)\geq (1+\delta)\frac{m\sigma^2}{n}+(1+\delta^{-1})\frac{\sigma^2x}{n}\right)\leq C(p)\frac{m}{x^{p/2}}.\]
Set 
\[Q_{m\vee m'}:=\left(\sup_{\substack{f\in \hat S_{m\vee m'}\\\|f\|_{\Gamma_n}=1}}\nu_n^2(f)-p(m,m')\right)_+.\]
We have, for all $m,m'\in\Mnr$
\begin{eqnarray*}
\E_\mathbf{X}\left[Q_{m\vee m'}\right]&=&\int_0^{+\infty}\Pb_\mathbf{X}\left(Q_{m\vee m'}\geq  t\right)dt\\
&\leq& C(p)\frac{\sigma^{p}}{n^{p/2}}\int_0^{+\infty}\frac{dt}{\left(t+\frac{\sigma^2m\vee m'}{n}(1+\delta)\right)^{p/2}}\leq C'(p,\delta)\frac{\sigma^2}{n}(m\vee m')^{1-p/2}.
\end{eqnarray*}
As $p>4$, $(m\vee m')^{1-p/2}\leq 1$ and we obtain the expected result.
\end{proof}

\subsection{Upper-bound of the risk on $\Delta_n^\complement$}

We first bound the probability of $\Delta_n^\complement$:
\begin{lemma}\label{deltanb}
Under assumptions~\textbf{H2}, \textbf{H3} and \textbf{H4} and if the decreasing rate of $(\lambda_j)_{j\geq 1}$ is given by \textbf{(P)} or \textbf{(E)}, the set $\Delta_n$ defined by Equation~(\ref{deltan}) verifies
\[\Pb(\Delta_n^\complement)\leq C/n^6,\]
with $C>0$ independent of $n$.
\end{lemma}
\begin{proof}
First remark that:
\[\Pb(\Delta_n^\complement)=\Pb(\Delta_n^\complement\cap \{\hat N_n\leq N_n\})+\Pb(\Delta_n^\complement\cap \{\hat N_n>N_n\}),\]
the second term of this equality is easily bounded by $Cn^{-6}$ by Lemma~\ref{lem:hNn}. It remains to bound the first term. We have
\begin{equation*}
\Delta_n^\complement\cap \{\hat N_n\leq N_n\}=\left\{\inf_{f\in\hat{\mathcal S}_n}\frac{\|f\|_{\Gamma_n}^2}{\|f\|_\Gamma^2}<\rho_0^{-1} \right\}\cap \{\hat N_n\leq N_n\}\subset\left\{\inf_{f\in \hat S_{N_n}}\frac{\|f\|_{\Gamma_n}^2}{\|f\|_\Gamma^2}<\rho_0^{-1} \right\}  .
\end{equation*}

Let $f=\sum_{j=1}^{\hat N_n}\alpha_j\hat\psi_j\in\hat{S}_{N_n}$, we have a.s.
\[\|f\|_{\Gamma_n}^2=\sum_{j=1}^{N_n}\hat\lambda_j\alpha_j^2=|\hat\Lambda_n\alpha|_2^2,\]
where $|\cdot|_2$ is the norm of $\R^{ N_n}$ defined by $|x|_2^2=\displaystyle{\sum_{j=1}^{N_n}x_i^2}$ for all $x=(x_1,...,x_n)'\in\R^{ N_n}$ and $\hat\Lambda_n$ is the diagonal matrix with diagonal entries $\{\sqrt{\hat\lambda_1},...,\sqrt{\hat\lambda_n}\}$.\\
Moreover
\[\|f\|_\Gamma^2=\sum_{j,k=1}^{N_n}\alpha_j\alpha_k<\Gamma^{1/2}\hat\psi_j,\Gamma^{1/2}\hat\psi_k>=\alpha'\Psi_n\alpha,\]
where $\Psi_n$ is the symmetric and positive-definite matrix
\[\Psi_n:=\left(<\Gamma^{1/2}\hat\psi_j,\Gamma^{1/2}\hat \psi_k>\right)_{1\leq j,k\leq \hat N_n}.\]
Then,
\[\frac{\|f\|_{\Gamma_n}^2}{\|f\|_\Gamma^2}=\frac{|\hat\Lambda_n\alpha|_2^2}{\alpha'\Psi_n\alpha}=\frac{|\hat\Lambda_n\alpha|_2^2}{|\Psi_n^{1/2}\alpha|_2^2}=\frac{|\hat\Lambda_{n}\Psi_n^{-1/2}\Psi_n^{1/2}\alpha|_2^2}{|\Psi_n^{1/2}\alpha|_2^2}.\]
Now
\[\inf_{f\in\hat S_{N_n}}\frac{\|f\|_{\Gamma_n}^2}{\|f\|_\Gamma^2}=\inf_{\alpha\in\R^{ N_n}\backslash\{0\}}\frac{\alpha'\Psi_n^{-1/2}\hat\Lambda_n^2\Psi_n^{-1/2}\alpha}{|\alpha|_2^2}=\min\{\lambda, \lambda\text{ eigenvalue of }\Psi_n^{-1/2}\hat\Lambda_n^2\Psi_n^{-1/2}\}.\]

On the set $\mathcal J_n$ defined by Equation~(\ref{Jnrec}) in Section~\ref{sec:perturbation}, by Lemma~\ref{inclu}, for all, $j=1,..., N_n$, $\hat\lambda_j>0$. Hence the matrix $\hat\Lambda_n$ is invertible, therefore
\[\inf_{f\in\hat{\mathcal S}_n}\frac{\|f\|_{\Gamma_n}^2}{\|f\|_\Gamma^2}=\rho\left((\Psi_n^{-1/2}\hat\Lambda_n^2\Psi_n^{-1/2})^{-1}\right)^{-1}=\rho(\hat\Lambda_n^{-1}\Psi_n\hat\Lambda_n^{-1})^{-1},\]
where, for a matrix $A$, $\rho(A)=\max\{|\lambda|, \lambda\text{ is a complex eigenvalue of }A\}$ denotes the spectral radius of $A$. We have then
\begin{equation}\label{Deltan1}
\Pb\left({\Delta}_n^\complement\cap\left\{\widehat{N_n}<N_n\right\}\right)\leq\Pb(\mathcal J_n\cap\{\rho(\hat\Lambda_n^{-1}\Psi_n\hat\Lambda_n^{-1})>\rho_0\})+\Pb({\mathcal J}_n^\complement).
\end{equation}

By Lemma~\ref{pbJn} in Section~\ref{sec:perturbation}, $\Pb({\mathcal{J}_n}^\complement)\leq C/n^6$, with $C$ depending only on $\Gamma$ and $b$. Thus it remains to control the spectral radius of $\hat\Lambda_n^{-1}\Psi_n\hat\Lambda_n^{-1}$.

We define a linear (random) application $O$ from $\R^{N_n}$ to $\Ld$ by:
\[O : \alpha=(\alpha_1,...,\alpha_{N_n})'\mapsto \sum_{j=1}^{ N_n}\alpha_j\hat\psi_j.\]
We denote by $O^*$ the adjoint of $O$, which is the linear map from $\Ld$ to $\R^{ N_n}$ defined by:
\[O^* : f\mapsto (<f,\hat\psi_j>_{1\leq j\leq  N_n}).\]
 We can check that $OO^*=\hat\Pi_{ N_n}$ and $\Psi_n$ is the matrix of the linear map $O^*\Gamma O$ in the standard basis of $\R^{ N_n}$.

It is known that the spectral radius of an operator is equal to the spectral radius of its adjoint, then,
\begin{equation}\label{spectral_transf}
\rho(\hat{\Lambda}_n^{-1}\Psi_n\hat{\Lambda}_n^{-1})=\rho(\hat{\mathcal{L}}_n^{-1}O^*\Gamma O\hat{\mathcal{L}}_n^{-1})=\rho(\Gamma^{1/2}O\hat{\mathcal{L}}_n^{-1}\hat{\mathcal{L}}_n^{-1}O^*\Gamma^{1/2}),
\end{equation}
where $\hat{\mathcal L}_n$ denotes the linear endormorphism of $\R^{N_n}$ whose matrix in the standard basis is $\hat\Lambda_n$. 
 Denote by $\Pi_{ N_n}$ the orthogonal projector onto $S_{ N_n}=\text{span}\{\psi_1,...,\psi_{ N_n}\}$. Moreover, let $\Gamma^\dag$ (resp. $\Gamma_n^\dag$) the pseudo-inverse of operator $\Gamma$ (resp.  $\Gamma_n$) on $\hat S_{ N_n}$ (resp. $\hat{\mathcal S}_n$), defined by:
 \begin{equation}\label{gammadag}\Gamma^\dag f:=\sum_{j=1}^{ N_n}\frac{<f,\psi_j>}{\lambda_j}\psi_j\text{ and }\Gamma_n^\dag f:=\sum_{j=1}^{ N_n}\frac{<f,\hat\psi_j>}{\hat\lambda_j}\hat\psi_j\mathbf{1}_{\{\hat\lambda_j>0\}},\end{equation}
 we have $\Gamma^{1/2}\Gamma^\dag\Gamma^{1/2}=\Pi_{N_n}$ and $O\hat{\mathcal L}_n^{-2}O^*=\Gamma_n^\dag$.
Then, 	\[\Gamma^{1/2}O\hat{\mathcal{L}}_n^{-1}\hat{\mathcal{L}}_n^{-1}O^*\Gamma^{1/2}=\Gamma^{1/2}\Gamma_n^\dag\Gamma^{1/2}=\Gamma^{1/2}(\Gamma^\dag+\Gamma_n^\dag-\Gamma^\dag)\Gamma^{1/2}=\Pi_{\hat N_n}+\Gamma^{1/2}(\Gamma_n^\dag-\Gamma^\dag)\Gamma^{1/2},
\]
and by Equation~(\ref{spectral_transf})
	\[\rho(\hat{\Lambda}_n^{-1}\Psi_n\hat{\Lambda}_n^{-1})=\|\Pi_{\hat N_n}+\Gamma^{1/2}(\Gamma_n^\dag-\Gamma^\dag)\Gamma^{1/2}\|_\infty\leq 1+\|\Gamma^{1/2}(\Gamma_n^\dag-\Gamma^\dag)\Gamma^{1/2}\|_\infty,
\]
where $\|\cdot\|_\infty$ denotes the usual operator norm.\\
Now
	\[\Pb(\mathcal J_n\cap\{\rho(\hat\Lambda_n^{-1}\Psi_n\hat\Lambda_n^{-1})>\rho_0\})\leq \Pb\left(\mathcal J_n\cap\left\{\|\Gamma^{1/2}(\Gamma_n^\dag-\Gamma^\dag)\Gamma^{1/2}\|_\infty>\rho_0-1\right\}\right).
\]
Thus, the results of Lemma~\ref{deltanb_tech} -- whose technical proof is given in Section~\ref{subsec:tech_DeltanC} -- in Equation~(\ref{Deltan1}) allows us to bound $\Pb\left({\Delta}_n^\complement\cap\left\{\widehat{N_n}<N_n\right\}\right)$. Then the proof is finished by Lemma~\ref{lem:hNn}.
\end{proof}

\begin{lemma}\label{deltanb_tech}Suppose that assumptions~\textbf{H2}, \textbf{H3} and \textbf{H4} are fulfilled and that the decreasing rate of $(\lambda_j)_{j\geq 1}$ is given by \textbf{(P)} or \textbf{(E)}, then
\[\Pb\left(\mathcal J_n\cap\left\{\|\Gamma^{1/2}(\Gamma_n^\dag-\Gamma^\dag)\Gamma^{1/2}\|_\infty>\rho_0-1\right\}\right)\leq C n^{-6},\]
with $C>0$ independent of $n$.
\end{lemma}

\begin{lemma}\label{T1}
For all $\beta\in\Ld$, if assumptions of Lemma~\ref{deltanb} are fulfilled then
\[\E[\|\widetilde\beta-\beta\|_\Gamma\mathbf 1_{\Delta_n^\complement}]\leq\frac Cn(1+\|\beta\|_\Gamma^2),\]
with $C>0$ independent of $\beta$ and $n$.

\end{lemma}
\begin{proof}
First remark that, as $Y_i=<\beta,X_i>+\varepsilon_i$, for all $m\in\widehat{\Mnr}$
\[\hat\beta_m=\sum_{j=1}^m\frac{1}{n}\sum_{i=1}^nY_i\frac{<X_i,\hat\psi_j>}{\hat\lambda_j}\hat\psi_j=\hat\Pi_m\beta+R_m,\]
where 
\[R_m:=\sum_{j=1}^m\frac{1}{n}\sum_{i=1}^n\varepsilon_i\frac{<X_i,\hat\psi_j>}{\hat\lambda_j}\hat\psi_j.\]
Then 
\begin{eqnarray*}
\E[\|\widetilde\beta-\beta\|_\Gamma\mathbf 1_{\Delta_n^\complement}]&\leq& 2\E\left[\|\hat\Pi_{\hat m}\beta-\beta\|_\Gamma^2\mathbf 1_{\Delta_n^\complement}\right]+2\E\left[\|R_{\hat m}\|_\Gamma^2\mathbf 1_{\Delta_n^\complement}\right]\leq 2\|\beta\|_\Gamma^2\Pb(\Delta_n^\complement)+2\E\left[\|R_{\hat m}\|_\Gamma^2\mathbf 1_{\Delta_n^\complement}\right].
\end{eqnarray*}
The first term can be easily bounded using the results of Lemma~\ref{deltanb}. Then we focus on the second term, the idea is to bound the quantity $\|R_{\hat m}\|_\Gamma$ by  $\|R_{\hat m}\|$ which can be written simply,
\[\E\left[\|R_{\hat m}\|_\Gamma^2\mathbf{1}_{\Delta_n^\complement}\right]\leq\rho(\Gamma)\E\left[\|R_{\hat m}\|^2\mathbf{1}_{\Delta_n^\complement}\right]=\rho(\Gamma)\E\left[\sum_{j=1}^{\hat m}<R_{\hat m},\hat\psi_j>^2\mathbf{1}_{\Delta_n^\complement}\right].\]
Now remark that, since $\hat m\leq \hat N_n\leq20\sqrt{n}$,
\[\sum_{j=1}^{\hat m}<R_{\hat m},\hat\psi_j>^2\leq\sum_{j=1}^{\hat m}\left(\frac{1}{n}\sum_{i=1}^n\varepsilon_i\frac{<X_i,\hat\psi_j>}{\hat\lambda_j}\right)^2\leq\mathfrak{s}_n^{-1} \sum_{j=1}^{20\lfloor\sqrt{n}\rfloor}\left(\frac{1}{n}\sum_{i=1}^n\varepsilon_i\frac{<X_i,\hat\psi_j>}{\sqrt{\hat\lambda_j}}\right)^2\mathbf{1}_{\{\hat\lambda_j>0\}}.\]

Then we have, by independence of $\varepsilon_i$ with $X_i$ and $\Delta_n$
\[\E\left[\|R_{\hat m}\|_\Gamma^2\mathbf{1}_{\Delta_n^\complement}\right]\leq20\rho(\Gamma) \frac{\sigma^2}{\sqrt n}\mathfrak{s}_n^{-1}\Pb(\Delta_n^\complement),\]
and the results come from lemmas~\ref{deltanb} and \ref{lem:hNn} and the definition of $\mathfrak{s}_n$. 
\end{proof}

\begin{lemma}\label{lem:hNn}If Assumption~\textbf{H2} is fulfilled and $n\geq 6$, then
\[\Pb(\hat N_n>N_n)\leq Cn^{-6},\]
with $C$ independent of $\beta$ and $n$.
\end{lemma}
\begin{proof}
The sequence $(\lambda_j)_{j\geq 1}$ being non-increasing we have
\[\Pb(\hat N_n>N_n)\leq\Pb(\lambda_{N_n+1}\geq \lambda_{\hat N_n})\leq\Pb\left(\left\{\lambda_{N_n+1}\geq \lambda_{\hat N_n}\right\}\cap\mathcal A_n\right)+\Pb\left(\mathcal A_n^\complement\right) ,\]
where $\mathcal{A}_n$ is defined by Equation~(\ref{An}) in Section~\ref{sec:perturbation}.
Then by definition of $\mathcal A_n$,
\[\left\{\lambda_{N_n+1}\geq \lambda_{\hat N_n}\right\}\cap\mathcal A_n\subset\left\{\lambda_{N_n+1}\geq \hat\lambda_{\hat N_n}-\frac{\delta_{\hat N_n}}{2}\right\}=\emptyset,\]
since $\lambda_{N_n+1}<n^{-2}$, $\hat\lambda_{N_n}\geq\mathfrak{s}_n$ and $\delta_{\hat N_n}\leq \frac{\lambda_{\hat N_n}}{4}\leq \frac{n^{-2}}{4}$.
Thus the proof is finished following the conclusions of Remark~\ref{rem:defJn} in Section~\ref{sec:perturbation}. 
\end{proof}
\subsubsection*{Acknowledgements}

The authors would like to warmly thank Dr Nicolas Verzelen whose helpful comments and suggestions have improved significantly the presentation of this work. 
\begin{appendices}
\section{Perturbation theory background}
\label{subsec:perturbation_theory}

Many intermediate results are based on perturbation theory. We give in this section some preliminary results on this subject. The aim is to control the proximity between the random space $\hat S_m$ spanned by the eigenfunctions of $\Gamma_n$ and the space $S_m$ spanned by the eigenfunctions of $\Gamma$. 

Recall that $\Pi_m$ (resp. $\hat\Pi_m$) denotes the orthonormal projector onto $S_m$ (resp. $\hat S_m$) and $\pi_j$ (resp. $\hat\pi_j$) denotes the orthonormal projector onto $\text{span}\{\psi_j\}$ (resp. $\text{span}\{\hat\psi_j\}$). We write the difference of projectors $\Pi_m-\hat\Pi_m$ (or equivalently $\pi_j-\hat\pi_j$) explicitly in terms of the operators difference $\Gamma-\Gamma_n$ easier to handle.
 
\subsection{Exponential inequalities}
In the proofs, we use the following version of Bernstein's Inequality:
\begin{lemma}[Birgé et Massart (1998)~\cite{birge_minimum_1998}]\label{bernstein} Let $Z_1$, ..., $Z_n$ be independent random variables satisfying the moments conditions
\[\frac{1}{n}\sum_{i=1}^n\E[|Z_i^m|]\leq \frac{m!}{2}v^2c^{m-2}, \text{ for all }m\geq 2,\]
for some positive constants $v$ and $c$. Then, for any positive $\varepsilon$,
\[\Pb\left(\frac{1}{n}\sum_{i=1}^n(Z_i-\E[Z_i])\geq \varepsilon\right)\leq\exp\left(\frac{-n\varepsilon^2/2}{v^2+c\varepsilon}\right).\]
\end{lemma}
We use also a version of Lemma~\ref{bernstein} for Hilbert valued random variables which can be found in Bosq~\cite{bosq_linear_2000}.
\subsection{Preliminary notions}
\label{sec:perturbation}

Let $\gamma$ be either the rectangular path given by Figure~\ref{rect_contour} or the union  (for $j=1,...,m$) of the circular paths $\partial\Omega_j$ of center $\lambda_j$ and radius $\delta_j$ represented in Figure~\ref{circles}.

\begin{figure}
\begin{tikzpicture}
\draw (-1,0)--(0,0) node[below left]{$0$}--(1.5,0)node[below]{$\lambda_{m+1}$}--(2.5,0) node[below]{$\lambda_m$}--(4,0)node[below]{$\cdots$}--(6,0)node[below]{$\lambda_1$}--(12,0)node[below]{$2\lambda_1$}--(13,0);
\draw (0,-2)--(0,-1)node[left]{$-2\lambda_1$}--(0,1)node[left]{$2\lambda_1$}--(0,2);
\draw[red, thick] (2,-1) rectangle (12,1);
\draw (1.5,-1pt)--(1.5,1pt);
\draw (2,-1pt)--(2,1pt);
\draw (2.5,-1pt)--(2.5,1pt);
\draw (4,-1pt)--(4,1pt);
\draw (6,-1pt)--(6,1pt);
\draw (12,-1pt)--(12,1pt);
\draw (-1pt,-1)--(1pt,-1);
\draw (-1pt,1)--(1pt,1);
\draw [red,->, thick] (5,-1)--(5.5,-1);
\draw [red,->, thick] (12,0.3)--(12,0.7);
\draw [red,<-, thick] (5,1)--(5.5,1);
\draw [red,->, thick] (2,-0.3)--(2,-0.7);
\draw [red] (8.5,1)--(9,1.5) node[above right]{\textcolor{red}{$\gamma$}};
\draw [<-] (2,0.2)--(2.25,0.2)node[above]{$\delta_m$}--(2.5,0.2);
\draw [->] (2,0.2)--(2.5,0.2);
\end{tikzpicture}
\caption{\label{rect_contour}Rectangular contour} 
\end{figure}
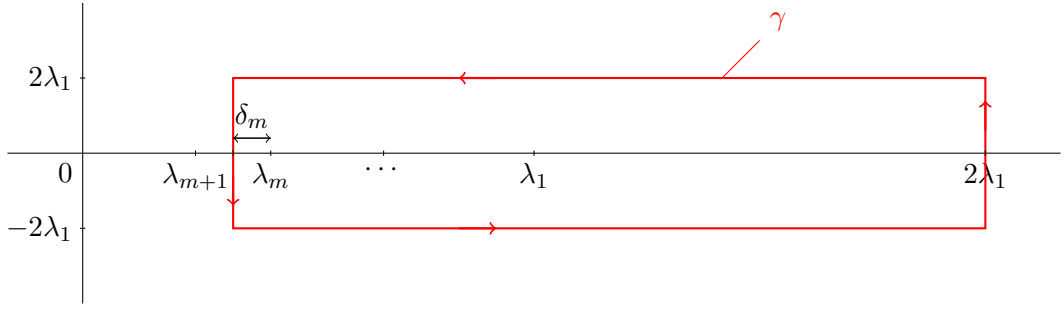

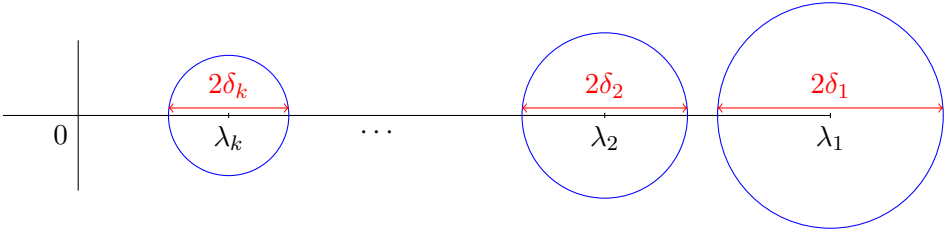
\begin{figure}
\begin{tikzpicture}
\draw (-1,0)--(0,0) node[below left]{$0$}--(2,0)node[below]{$\lambda_{k}$}--(4,0)node[below]{$\cdots$}--(7,0) node[below]{$\lambda_2$}--(10,0)node[below]{$\lambda_1$};
\draw (0,-1)--(0,1);
\draw (10,0)[blue] circle (1.5cm);
\draw (7,0)[blue] circle (1.1cm);
\draw (2,0)[blue] circle (0.8cm);
\draw [red,<->] (1.2,0.1)--(2,0.1)node[above]{$2\delta_{k}$}--(2.8,0.1);
\draw [red,<->] (5.9,0.1)--(7,0.1)node[above]{$2\delta_{2}$}--(8.1,0.1);
\draw [red,<->] (8.5,0.1)--(10,0.1)node[above]{$2\delta_{1}$}--(11.5,0.1);
\draw (10,-1pt)--(10,1pt);
\draw (7,-1pt)--(7,1pt);
\draw (2,-1pt)--(2,1pt);
\end{tikzpicture}
\caption{\label{circles}Contour made of disjoint circles} 
\end{figure}

We have, for all $j\geq 1$, $x\in[0,1]$:
\[\Pi_m\psi_j(x)=\mathbf{1}_{j\leq m}\psi_j(x)=\frac{1}{2i\pi}\int_{\gamma}\frac{\psi_j(x)}{\zeta-\lambda_j}d\zeta=\frac{1}{2i\pi}\int_{\gamma}(\zeta I-\Gamma)^{-1}\psi_j(x)d\zeta,\]
and
\begin{equation}
\label{pij}\Pi_m=\frac{1}{2i\pi}\int_{\gamma}(\zeta I-\Gamma)^{-1}d\zeta,
\end{equation}
 we refer to Chapter III of Dunford and Schwartz~\cite{dunford_schwartz} for an exact definition and properties of this integral.
 
 Now the aim is to write similarly the random projector $\hat\Pi_m$. This can be done if, for all $j\leq m$, $\hat\lambda_j$ is in the interior of $\gamma$. 
 
 For $t>0$, let $\mathcal J_\gamma(t)$ be the set
 \begin{equation}\label{defJn}
 \mathcal{J}_\gamma(t)=\left\{\sup_{\zeta\in\text{supp}(\gamma)}\left\|\mathcal{T}(\zeta)\right\|_\infty<t\right\},
 \end{equation}
 where $\mathcal{T}(\zeta):=R^{1/2}(\zeta)(\Gamma_n- \Gamma)R^{1/2}(\zeta)$ with $R(\zeta)=(\zeta I-\Gamma)^{-1}$.
Define also the set
\begin{equation}\label{An}
\mathcal A_n:=\bigcap_{j=1}^{n}\left\{|\hat\lambda_j-\lambda_j|<\frac{\delta_j}{2}\right\}.
\end{equation}

 \begin{lemma}
\label{inclu}Let $t<1/2$, then for both circular and rectangular path $\gamma$ $\mathcal{J}_\gamma(t)
\subset\mathcal{A}_{n}$. 
\end{lemma}
\noindent The proof of Lemma~\ref{inclu} may be deduced from the proof of Lemma~14, p.12 of Mas and Ruymgaart~\cite{mas_high_2012}. 

\hspace{0.2cm}
Now Lemma~\ref{inclu}, allows us to write that for all $t<1/2$:
\begin{equation}\label{hpij}
\hat\Pi_m=\mathbf 1_{\mathcal J_\gamma(t)}\frac{1}{2i\pi}\int_{\gamma}(\zeta I-\Gamma_n)^{-1}d\zeta. 
\end{equation}
Define $\widehat{R}\left(  \zeta\right)  :=\left(  \zeta
I-\Gamma_{n}\right)  ^{-1}$, equations~(\ref{pij}) and~(\ref{hpij}) lead to:
  \begin{equation*}
  (\hat\Pi_m-\Pi_m)\mathbf 1_{\mathcal J_\gamma(t)}=\mathbf 1_{\mathcal J_\gamma(t)}\frac{1}{2i\pi}\int_{\gamma}\left(R\left(\zeta\right)-\widehat{R}\left(\zeta\right)\right)d\zeta, \text{ for all  }t<1/2.
  \end{equation*} 
   Then we rewrite the interior of the last integral by remarking that
   \begin{equation}\label{diff1}
   R\left(\zeta\right)-\widehat{R}\left(\zeta\right)=\widehat{R}\left(\zeta\right)\left(\Gamma-\Gamma_n\right)R\left(\zeta\right)=\widehat{R}\left(\zeta\right)\left(\zeta I-\Gamma\right)^{1/2}\mathcal{T}\left(\zeta\right)R^{1/2}\left(\zeta\right).
   \end{equation} 
  
By definition, when $ t<1$, on the set $\mathcal J_{\gamma}(t)$, the operator  $I-\mathcal{T}(\zeta)$ is invertible for all $\zeta\in\text{supp}(\gamma)$ and we have:
\[\left(I-\mathcal{T}\left(\zeta\right)\right)^{-1}=\left(\zeta I-\Gamma\right)^{1/2}\widehat{R}\left(\zeta\right)\left(\zeta I-\Gamma\right)^{1/2}.\]
Then Equation~(\ref{diff1}) leads to
\begin{equation*}
\left(\widehat{R}\left(  \zeta\right)  -R\left(  \zeta\right)\right)\mathbf 1_{\mathcal J_\gamma(t)}  =R^{1/2}\left(
\zeta\right)  \left[  I-\mathcal{T}\left(  \zeta\right)  \right]
^{-1}\mathcal{T}\left(  \zeta\right)  R^{1/2}\left(  \zeta\right)\mathbf 1_{\mathcal J_\gamma(t)}. 
\end{equation*}

We obtain in a similar way a rewriting of $(\hat\pi_j-\pi_j)\mathbf 1_{\mathcal J_\gamma(t)}$ or $(\Gamma^\dag-\Gamma_n^\dag)\mathbf 1_{\mathcal J_\gamma(t)}$ where $\Gamma^\dag$ and $\Gamma_n^\dag$ are defined by Equation~(\ref{gammadag}). All results are summarized in the following lemma.
\begin{lemma}\label{diff_proj}For all $t<1/2$ if $\gamma$ is either the union of circular contour represented in Figure~\ref{circles} or the rectangular contour given by Figure~\ref{rect_contour}: 
\begin{eqnarray*}
(\hat\Pi_m-\Pi_m)\mathbf 1_{\mathcal J_\gamma(t)}&=&\frac{1}{2i\pi}\int_\gamma R^{1/2}\left(
\zeta\right)  \left[  I-\mathcal{T}\left(  \zeta\right)  \right]
^{-1}\mathcal{T}\left(  \zeta\right)  R^{1/2}\left(  \zeta\right)d\zeta\mathbf 1_{\mathcal J_\gamma(t)}\\
(\hat\pi_j-\pi_j)\mathbf 1_{\mathcal J_\gamma(t)}&=&\frac{1}{2i\pi}\int_{\partial\Omega_j} R^{1/2}\left(
\zeta\right)  \left[  I-\mathcal{T}\left(  \zeta\right)  \right]
^{-1}\mathcal{T}\left(  \zeta\right)  R^{1/2}\left(  \zeta\right)d\zeta\mathbf 1_{\mathcal J_\gamma(t)}\\
(\Gamma^\dag-\Gamma_n^\dag)\mathbf 1_{\mathcal J_\gamma(t)}&=&\frac{1}{2i\pi}\int_\gamma\frac 1\zeta R^{1/2}\left(
\zeta\right)  \left[  I-\mathcal{T}\left(  \zeta\right)  \right]
^{-1}\mathcal{T}\left(  \zeta\right)  R^{1/2}\left(  \zeta\right)d\zeta\mathbf 1_{\mathcal J_\gamma(t)}.
\end{eqnarray*}
\end{lemma}

The last lemma allows us to finally control our quantities on $\mathcal J_\gamma(t)^\complement$.
\begin{lemma}\label{pbJn}
Denote by
\begin{equation*} 
\mathbf{a}_{k}:= \sum_{i\neq k}\frac
{\lambda_{i}}{\left\vert \lambda_{i}-\lambda_{k}\right\vert }+\frac
{\lambda_{k}}{\delta_{k}}, \text{ for all }k\geq 1.
\end{equation*} 
If $\gamma$ is the path $\partial\Omega_k$ covering the circle of center $\lambda_k$ and of radius $\delta_k$ or the rectangular contour given in Figure~\ref{rect_contour} we have, under Assumption~\textbf{H2} :%
\[
\mathbb{P}\left(  \mathcal J_{\gamma}^\complement(t) \right)
\leq2\exp\left(  -\frac{nt^2}{2\mathbf{a}_k^{2}(2b-1)}\frac
{1}{(2b-1) +256b^3/((2b-1)  \mathbf{a}_{k}t}%
\right).
\]
\end{lemma}
The proof of Lemma~\ref{pbJn} relies on Bernstein's exponential inequality for Hilbert-valued random variables (see for instance Bosq~\cite{bosq_linear_2000}). Details can be found in Mas and Ruymgaart~\cite[Lemma 13]{mas_high_2012}. 
\begin{definition}Let, for all $t>0$, $\mathcal J_\gamma (t)$ be the set defined by Equation~(\ref{defJn}) then we define a set $\mathcal J_n$ in the following way: take $l_{j,n}:=\min\left\{\frac{\mathbf{a}_{j}}{\sqrt n}\ln n,1/2\right\}$ ,
\begin{description} 
\item[Circular contour]
\begin{equation}\label{Jncirc}
\mathcal J_n:=\bigcap_{j=1}^m\mathcal{J}_{\partial\Omega_j}(l_{j,n});
\end{equation}
\item[Rectangular contour]
\begin{equation}\label{Jnrec}
\mathcal J_n:=\mathcal{J}_\gamma(l_{m,n}).
\end{equation}
\end{description}
\end{definition}
\hspace{1cm}
\begin{remark}\label{rem:defJn}By Lemma~\ref{pbJn}, the sets $\mathcal J_n$ defined by Equation~(\ref{Jncirc}) or~(\ref{Jnrec}) verifies
\begin{equation}\label{PbJn}\Pb( \mathcal J_n^\complement)\leq 2\exp(-c^*\ln^2 n),\end{equation}
where $c^*>0$ depends only on $b$ and $c_\delta$. Moreover, as $l_{j,n}<1/2$, by Lemma~\ref{inclu}, we have 
\[\mathcal J_n\subset \mathcal A_n\]
and the results of Lemma~\ref{diff_proj} are true. 
\end{remark}

\subsection{Upper-bound on the distance between empirical and theoretical projectors}
We need a preliminary lemma
\begin{lemma}[Hilgert \textit{et al.}~\cite{hilgert_minimax_2012}, Lemma 10.1]
\label{trick1}If Assumption~\textbf{H4} is verified, then for all $k\in\N^*$ :%
\begin{equation*}
\mathbf a_k\leq C(\gamma)k\ln k%
\end{equation*}

\end{lemma}

\begin{lemma}\label{normnGamma}Let $r,R>0$ and $\beta\in\mathcal{W}_r^R$. Suppose that  assumptions \textbf{H2}, \textbf{H3} and \textbf{H4} are fulfilled. If $(\lambda_j)_{j\geq 1}$ decreases at polynomial rate (\textbf{P}) then
\begin{eqnarray*}
\E[\|\hat\Pi_{m}\beta-\Pi_{m}\beta\|_\Gamma^2\mathbf 1_{{\mathcal J}_n}]&\leq &C_1\frac{\ln^3 m}{n}m^{\max\{(1-r)_+,2-a-r\}}+C_2\frac{\ln^5 m\ln^4 n}{n^2}m^{\max\{(2-a+(7-r)_+)_+,2-a+(5-r)_+\}},
\end{eqnarray*}
and if $(\lambda_j)_{j\geq 1}$ decreases at exponential rate (\textbf{E})
\[\E[\|\hat\Pi_{m}\beta-\Pi_{m}\beta\|_\Gamma^2\mathbf 1_{{\mathcal J}_n}]\leq C_2\frac{\ln^3 m}{n}m^{(1-r)_+}+\frac{\ln^6 m \ln^4 n}{n^2},\]
with $C_1>0$, $C_2>0$ and $C_3>0$ depending on $r$, $R$ and on the sequence $(\lambda_j)_{j\geq 1}$ but are independent of $m$ and $n$.

\end{lemma}
\begin{proof}
First, by Equation~(\ref{PbJn})
\[\E[\|\hat\Pi_{m}\beta-\Pi_{m}\beta\|_\Gamma^2\mathbf 1_{{\mathcal J}_n^\complement}]\leq 2\|\beta\|_\Gamma^2\Pb({\mathcal J}_n^\complement)\leq C\|\beta\|_\Gamma^2/n^2,\]
with $C$ independent of $\beta$, $n$ and $m$. Now by Lemma \ref{diff_proj},
\begin{eqnarray*}
	(\hat\Pi_{m}\beta-\Pi_{m}\beta)\mathbf 1_{\mathcal J_n}=\mathbf{1}_{\mathcal J_n}\frac{1}{2i\pi}\sum_{j=1}^m\int_{\partial\Omega_{j}}R^{1/2}\left(  z\right)
\left[  I-\mathcal{T}\left(  z\right)  \right]  ^{-1}\widetilde{\Pi
}\left(  z\right)  R^{1/2}\left(  z\right)\beta  dz.
\end{eqnarray*}
Remark that $(I-\mathcal{T}(z))^{-1}\mathcal{T}(z)=\mathcal{T}(z)+(I-\mathcal{T}(z))^{-1}\mathcal{T}^2(z)$, then
\[\mathbf{1}_{\mathcal J_n}\left(\widehat{\Pi}_{m}-\Pi_{m}\right)   =A_{n}+B_{n},\]
with
\begin{align*}
A_{n}  &  =\mathbf{1}_{\mathcal J_n}\dfrac{1}{2\pi i}\sum_{j=1}^{m}\int_{\partial\Omega_{j}}\left[
\left(  zI-\Gamma\right)  ^{-1}\left(  \Gamma_{n}-\Gamma\right)  \left(
zI-\Gamma\right)  ^{-1}\right]  dz\\
B_{n}  &  =\mathbf{1}_{\mathcal J_n}\dfrac{1}{2\pi i}\sum_{j=1}^{m}\int_{\partial\Omega_{j}}\left[
\left(  zI-\Gamma\right)  ^{-1/2}\left[  I-\mathcal{T}\left(  z\right)  \right]  ^{-1}\mathcal T\left( z\right)^2 \left(  zI-\Gamma\right)  ^{-1/2}\right]  dz.
\end{align*}
We deal with $A_{n}$ first. Calculations show that%
\begin{align}
\mathbb{E}\left\Vert \Gamma^{1/2}A_{n}\beta\right\Vert ^{2}  &  \leq\frac
{1}{n}\sum_{j=1}^{m}\beta_{j}^{2}\sum_{k>m}\frac{\lambda_{j}\lambda
_{k}^{2}}{\left(  \lambda_{j}-\lambda_{k}\right)  ^{2}}+\frac{1}{n}\sum
_{j=1}^{m}\lambda_{j}^{2}\sum_{k>m}\frac{\lambda_{k}\beta_{k}^{2}}{\left(
\lambda_{j}-\lambda_{k}\right)  ^{2}}\nonumber\\
&  \leq\frac{1}{n}\sum_{j=1}^{m}\lambda_{j}\beta_{j}^{2}\mathbf{a}_j^2+\frac{1}{n}\sum_{j=1}^{m}\frac{\lambda_{j}^{2}}{(\lambda_j-\lambda_{m+1})^2}\lambda_{m+1}m^{-r}\sum_{k>m}k^r\beta_k^2\nonumber\\
&\label{Anb}\leq C\left( \frac{m^{(1-r)_+}\ln^3 m}{n}+\frac{m^{2-r}\ln^2m}{n}\lambda_{m+1}\right),
\end{align}
as soon as $(\lambda_j)_{j\geq 1}$ decreases exponentially or polynomially with $C>0$ depending only on $R$, $r$, $a$ and $\Gamma$. The last line comes from the inequality $\lambda_j\lesssim j^{-1}$ and Lemma~\ref{trick1}.
We turn now to $B_n$
\begin{eqnarray}
\|B_n\beta\mathbf 1_{\mathcal J_n}\|_\Gamma^2&&\nonumber\\
&&\label{diffproj}\hspace{-3cm}\leq \mathbf{1}_{\mathcal J_n}\frac{1}{4\pi^2}\sum_{k\geq1}\lambda_k\left(\sum_{j=1}^m\int_{\partial\Omega_{j}}<R^{1/2}\left(  \zeta\right)
\left[  I-\mathcal{T}\left(  \zeta\right)  \right]  ^{-1}\mathcal T\left(  \zeta\right)^2  R^{1/2}\left(  \zeta\right)\beta,\psi_k>d\zeta\right)^2.
\end{eqnarray}
We have:
\begin{eqnarray*}
<R^{1/2}\left(  \zeta\right)
\left[  I-\mathcal{T}\left(  \zeta\right)  \right]  ^{-1}\mathcal T\left(  \zeta\right)^2  R^{1/2}\left(  \zeta\right)\beta,\psi_k>=<
\left[  I-\mathcal{T}\left(  \zeta\right)  \right]  ^{-1}\mathcal T\left(  \zeta\right)^2  R^{1/2}\left(  \zeta\right)\beta,R^{1/2}\left(  \zeta\right)\psi_k>.
\end{eqnarray*}
We denote by $\beta^\diamond:=\sum_{j\geq 1}j^{r/2}\beta_j\psi_j$ ; as $\beta$ is in $\mathcal W_r^R$, the function $\beta^\diamond$ is in $\mathbb L^2([0,1])$. Moreover we denote by $P_r$ the diagonal compact operator defined by $P_r\psi_j=j^{-r/2}\psi_j$, we remark that $\beta=P_r\beta^\diamond$. We have
\[R^{1/2}(\zeta)\psi_k=(\zeta I-\Gamma)^{-1/2}\psi_k=\frac{1}{\sqrt{\zeta-\lambda_k}}\psi_k.\]

Then,
\begin{eqnarray*}
\vert<R^{1/2}\left(  \zeta\right)
\left[  I-\mathcal{T}\left(  \zeta\right)  \right]  ^{-1}\widetilde{\Pi
}\left(  \zeta\right)^2  R^{1/2}\left(  \zeta\right)\beta,\psi_k>\vert
\leq\frac{1}{\sqrt{|\zeta-\lambda_k|}}\left\|(I-\mathcal{T}(\zeta))^{-1}\right\|_\infty\left\|\mathcal{T}(\zeta)\right\|_\infty^2\left\|R^{1/2}(\zeta)P_r\right\|_\infty\left\|\beta^\diamond\right\|.
\end{eqnarray*}
Now on the set $\mathcal{J}_n$, by definition, we have for all $\zeta\in\Omega_j$:
\[\left\|(I-\mathcal{T}(\zeta))^{-1}\right\|_\infty<2\text{ and }\left\|\mathcal{T}(\zeta)\right\|_\infty< \frac{\mathbf{a}_j}{\sqrt{n}}\ln n.\]
Moreover, the eigenvalues of the operator $R^{1/2}(\zeta)P_r$ are $\{k^{-r/2}(\zeta-\lambda_k)^{1/2}, k\geq 1\}$ then, for all, $\zeta\in\partial\Omega_j$ : 
\begin{equation}\label{R12Pr}\left\|R^{1/2}(\zeta)P_r\right\|_\infty=\sup_{k\geq 1}\{k^{-r/2}|\zeta-\lambda_k|^{-1/2}\}= j^{-r/2}/\sqrt{\delta_j},\end{equation}
and Equation~(\ref{diffproj}) becomes: 
\begin{eqnarray*}
\|B_n\beta\mathbf 1_{\mathcal J_n}\|_\Gamma^2&\leq& \mathbf{1}_{\mathcal J_n}\frac{1}{\pi^2}\|\beta^\diamond\|^2\sum_{k\geq1}\lambda_k\left(\sum_{j=1}^m\frac{\mathbf{a}_j^2}{n}\ln^2 n \frac{j^{-r/2}}{\sqrt{\delta_j}}\int_{\partial\Omega_{j}}\frac{dz}{\sqrt{|z-\lambda_k|}}\right)^2.\\
\end{eqnarray*}

In the polynomial case \textbf{(P)}, calculations lead to the following bound 
 \begin{equation*}\|B_n\beta\mathbf 1_{\mathcal J_n}\|_\Gamma^2\leq C \frac{\ln^5 m\ln^4 n}{n^2}m^{\max\{(2-a+(7-r)_+)_+,2-a+(5-r)_+\}},\end{equation*}
 with $a>1$ such that $\lambda_j\asymp j^{-a}$ and $C>0$ depends only on $R$, $r$, $a$ and $\Gamma$. Gathering with~(\ref{Anb}) we obtain the expected result.
 
In the exponential case \textbf{(E)} we have
 \begin{equation*}\|B_n\beta\mathbf 1_{\mathcal J_n}\|_\Gamma^2\leq C' \frac{\ln^6 m \ln^4 n}{n^2},\end{equation*}
 with $C$ depending only on $R$, $r$, $a$ and $\Gamma$.

\end{proof}

 \subsection{Empirical and theoretical bias terms}
\begin{lemma}\label{control_biais}Suppose that assumptions~\textbf{H2}, \textbf{H3} and \textbf{H4} are fulfilled and that $\beta\in\mathcal W_r^R$ with $r,R>0$ such that, in the polynomial case \textbf{(P)}, $a+r>2$. Then for all $m=1,...,N_n$:
\begin{equation}\label{bound}\E[\|\beta-\hat\Pi_m\beta\|_n^2]\leq 4\E[\|\beta-\hat\Pi_m\beta\|_\Gamma^2]+\tau_{m, n},\end{equation}
where $\tau_{m, n}\leq m\varepsilon(m)/n$ where $\varepsilon(m)\rightarrow 0$ when $m\rightarrow+\infty$ and is independent of $\beta$ and $m$.
\end{lemma}
\begin{proof}
First imagine that the random projectors in the equation above are replaced by
non random one. It is elementary to see that%
\[
\mathbb{E}\left[\left\Vert \beta-\Pi_{m}\beta\right\Vert _{n}^{2}\right]=\left\Vert
\beta-\Pi_{m}\beta\right\Vert _{\Gamma}^{2}%
\]
and that consequenlty to get (\ref{bound}) it is enough to show that both
$\mathbb{E}\left[\left\Vert \left(  \Pi_{m}-\widehat{\Pi}_{m}\right)  \beta
\right\Vert _{\Gamma}^{2}\right]$ and \\$\mathbb{E}\left[\left\Vert \left(  \Pi_{m}%
-\widehat{\Pi}_{m}\right)  \beta\right\Vert _{n}^{2}\right]$ are bounded by
$\tau_{m, n}$. The first bound was proved asymptotically in Cardot~\textit{et al.}~\cite{cardot_clt_2007} and non-asymptotically in the Proposition 20 of Crambes and Mas~\cite{crambes_optimal_2012} in a slightly more
general framework. Specifically these authors get%
\[
\mathbb{E}\left[\left\Vert \left(  \Pi_{m}-\widehat{\Pi}_{m}\right)  \beta
\right\Vert _{\Gamma}^{2}\right]\leq A\frac{m^{2}\lambda_{m}}{n}%
\]
where $A$ does not depend on $m$ and $n$. The only point left is to prove the
same sort of bound for 
\[\mathbb{E}\left[\left\Vert \left(  \Pi_{m}-\widehat{\Pi}%
_{m}\right)  \beta\right\Vert _{n}^{2}\right].\]
 The derivation makes use of
perturbation methods already  used in other parts of proofs. We will skip
technical details to concentrate on the essential facts.

In a first step remark that
\[\left\Vert \left(  \Pi_{m}-\widehat{\Pi}_{m}\right)
\beta\right\Vert _{n}^{2}=\left\langle \left(  \Gamma_{n}-\Gamma\right)
\left(  \Pi_{m}-\widehat{\Pi}_{m}\right)  \beta,\left(  \Pi_{m}-\widehat{\Pi
}_{m}\right)  \beta\right\rangle +\left\Vert \left(  \Pi_{m}-\widehat{\Pi}%
_{m}\right)  \beta\right\Vert _{\Gamma}^{2}\]
 and it is enough to focus on the
first term and to prove the bound for%
\[
\mathbb{E}\left[\left\Vert \left(  \Gamma_{n}-\Gamma\right)  \left(  \Pi
_{m}-\widehat{\Pi}_{m}\right)  \beta\right\Vert\right].
\]
after a Cauchy-Schwartz's Inequality coupled with the fact that $\left\Vert
\left(  \Pi_{m}-\widehat{\Pi}_{m}\right)  \beta\right\Vert \leq2\left\Vert
\beta\right\Vert .$ 

Now by Lemma~\ref{diff_proj}
\begin{equation*}
\left(  \Gamma_{n}-\Gamma\right)  \left(  \Pi_{m}-\widehat{\Pi}_{m}\right)
\beta \mathbf 1_{\mathcal J_n} =\frac{1}{2\pi i}\hspace{-1mm}\sum_{j=1}^{m}\hspace{-1mm}\int_{\partial\Omega_{j}}\hspace{-5mm}\left(
\Gamma_{n}-\Gamma\right)\hspace{-1mm}  R^{1/2}\left(  \zeta\right)\hspace{-1mm}  \left[  I-\mathcal T\left(  \zeta\right)  \right]  ^{-1}\hspace{-1mm}\mathcal{T}\left(  \zeta\right)
\hspace{-1mm}R^{1/2}\left(  \zeta\right) \hspace{-0.8mm} \beta d\zeta\mathbf 1_{\mathcal J_n}.
\end{equation*}
Hence by definition of $\mathcal{J}_n$, $\|(I-\mathcal{T}(\zeta))^{-1}\|_\infty\mathbf 1_{\mathcal J_n}\leq 2$ and $\|\mathcal{T}(\zeta)\|_\infty\mathbf 1_{\mathcal J_n}\leq \frac{\mathbf{a}_j\ln n}{n}$ then
\begin{align}
\mathbb{E}\left[\left\Vert \left(  \Gamma_{n}-\Gamma\right)  \left(  \Pi
_{m}-\widehat{\Pi}_{m}\right)  \beta\right\Vert\mathbf 1_{\mathcal J_n} \right]
& \leq \frac{1}{\pi}\sum_{j=1}^{m}%
\int_{\partial\Omega_{j}}\mathbb{E}\left[  \left\Vert \left(  \Gamma
_{n}-\Gamma\right)  R^{1/2}\left(  \zeta\right)  \right\Vert _{\infty
}\left\Vert \mathcal{T}\left(  \zeta\right)  \right\Vert _{\infty}\mathbf 1_{\mathcal J_n}\right]
\left\Vert R^{1/2}\left(  \zeta\right)  \beta\right\Vert d\zeta\nonumber\\
&\leq\frac{\ln n}{\pi\sqrt{n}}\sum_{j=1}^{m}\textbf{a}_{j}\int_{\partial\Omega_{j}%
}\left\Vert R^{1/2}\left(  \zeta\right)  P_{r}\right\Vert _{\infty}\left\Vert
\beta^{\Diamond}\right\Vert \mathbb{E}\left[\left\Vert \left(  \Gamma_{n}%
-\Gamma\right)  R^{1/2}\left(  \zeta\right)  \right\Vert _{\infty}\mathbf 1_{\mathcal J_n}\right]d\zeta\nonumber\\
& \leq\frac{\ln n}{\pi\sqrt{n}}\left\Vert \beta^{\Diamond}\right\Vert
\sum_{j=1}^{m}\textbf{a}_{j}\frac{j^{-r/2}}{\sqrt{\delta_{j}}}\int_{\partial\Omega_{j}%
}\sqrt{\mathbb{E}\left[\left\Vert \left(  \Gamma_{n}-\Gamma\right)  R^{1/2}\left(
\zeta\right)  \right\Vert _{HS}^{2}\mathbf 1_{\mathcal J_n}\right]}d\zeta\label{control_lemme12}
\end{align}
where we recall that, by Equation~(\ref{R12Pr}), \[\left\Vert R^{1/2}\left(  \zeta\right)  P_{r}\right\Vert _{\infty}\leq \frac{j^{-r/2}}{\sqrt{\delta_{j}}}\]
 (the definitions of $\beta^\diamond$ and $P_r$ are given in the proof of Lemma~\ref{normnGamma}).

Treating $\mathbb{E}\left[\left\Vert \left(  \Gamma_{n}-\Gamma\right)
R^{1/2}\left(  \zeta\right)  \right\Vert _{HS}^{2}\right]$ with computations similar
to those carried previously we get, for all $\zeta\in\partial\Omega_j$%
\[
\mathbb{E}\left[\left\Vert \left(  \Gamma_{n}-\Gamma\right)  R^{1/2}\left(
\zeta\right)  \right\Vert _{HS}^{2}\right]\leq\frac{C'\mathbf a_{j}}{n},%
\]%
with $C'=\text{Tr}(\Gamma)\max\{1,b-1\}$ and putting into Equation~(\ref{control_lemme12}) we obtain
\[
\mathbb{E}\left[\left\Vert \left(  \Gamma_{n}-\Gamma\right)  \left(  \Pi
_{m}-\widehat{\Pi}_{m}\right)  \beta\right\Vert\mathbf 1_{\mathcal J_n}\right] \leq\frac{C'\ln n}%
{\pi n}\left\Vert \beta^{\Diamond}\right\Vert \sum_{j=1}^{m}\mathbf a_{j}^{3/2}%
j^{-r/2}\sqrt{\delta_{j}}.%
\]
Considering again two cases related to the rate of decrease for the
eigenvalues we see first that for an exponential decay the term above is
bounded up to a constant by $ (\ln n) /n.$ Secondly in case of
polynomial decay we get :%
\begin{align*}
\mathbb{E}\left[\left\Vert \left(  \Gamma_{n}-\Gamma\right)  \left(  \Pi
_{m}-\widehat{\Pi}_{m}\right)  \beta\right\Vert\mathbf 1_{\mathcal J_n}\right]  & \leq\frac{C'\ln n}%
{\pi n}\left\Vert \beta^{\Diamond}\right\Vert \sum_{j=1}^{m}j^{3/2}j^{-r/2}%
j^{-\left(  1+a\right)  /2}\ln^{3/2} j\\
& \leq\frac{C'\ln n}{\pi n}\left\Vert \beta^{\Diamond}\right\Vert m^{2-\left(
r+a\right)  /2}\ln^{5/2} m%
\end{align*}
and $\ln n\ \ln^{5/2} m\cdot m^{2-\left(  r+a\right)  /2}/n=o\left(  m/n\right)  $ when
$r+a>2$.

Thus the proof is finished by Lemma~\ref{pbJn} 
\[\E\left[\left\Vert \left(  \Pi_{m}-\widehat{\Pi}_{m}\right)
\beta\right\Vert _{n}^{2}\mathbf 1_{\mathcal{J}_n^\complement}\right]\leq\|\beta\|_\Gamma^2\Pb(\mathcal J_n^\complement)\leq \frac{C''(b,\Gamma)}{n^2}.\] 

\end{proof}
\subsection{Technical part of the bound on $\Pb(\Delta_n^\complement)$}
\label{subsec:tech_DeltanC}
\begin{proof}[Proof of lemma~\ref{deltanb_tech}]

Let $\gamma$ be the contour defined by Figure~\ref{rect_contour} with $m=N_n$.

We have by Lemma~\ref{diff_proj} and the fact that $(I-\mathcal{T}(z))^{-1}\mathcal{T}(z)=\mathcal{T}(z)+(I-\mathcal{T}(z))^{-1}\mathcal{T}^2(z)$ 
\begin{align}
\Gamma^{1/2}\left[  \Gamma_{n}^{\dag}-\Gamma^{\dag}\right]
\Gamma^{1/2}\mathbf{1}_{\mathcal J_n} &  =\mathbf{1}_{\mathcal J_n}\frac{1}{2i\pi}\int_{\gamma}\frac{1}{z}\Gamma^{1/2}R^{1/2}(z)\left[I-\mathcal{T}(z)\right]^{-1}\mathcal{T}(z)R^{1/2}(z)
\Gamma^{1/2}dz\label{forme1}\\
&\hspace{-2cm}  =\mathbf{1}_{\mathcal J_n}\frac{1}{2i\pi}\int_{\gamma}\frac{1}{z}\Gamma^{1/2}R(z)\left(  \Gamma_{n}-\Gamma\right)  R(z)
\Gamma^{1/2}dz\nonumber\\
&\hspace{-1.5cm} +\mathbf{1}_{\mathcal J_n}\frac{1}{2i\pi}\int_{\gamma}\frac{1}{z}\Gamma^{1/2}R^{1/2}(z)\left[I-\mathcal{T}(z)\right]^{-1}\mathcal{T}^2(z)R^{1/2}(z)\Gamma^{1/2}dz.\label{forme2}
\end{align}

Now, we consider separately the two decreasing rates of the $\lambda_j$'s. 
\paragraph{Exponential decrease : $\lambda_j\asymp \exp(-j^a)$, with $a>0$}

By Equation~(\ref{forme1}) and the fact that $\|(I-\mathcal{T}(z))^{-1}\|_\infty<2$, on the set $\mathcal{J}_n$
\begin{align*}
\|\Gamma^{1/2}(\Gamma_n^\dag-\Gamma^\dag)\Gamma^{1/2}\|_\infty\mathbf{1}_{\mathcal J_n}&\leq \frac{1}{2\pi} \left\Vert \int_{\gamma}\frac{1}{z}\Gamma^{1/2}R^{1/2}(z)\left[I-\mathcal{T}(z)\right]^{-1}\mathcal{T}(z)R^{1/2}(z)\Gamma^{1/2}dz\right\Vert
_{\infty}\hspace{-3mm}\mathbf{1}_{\mathcal J_n}\\
&  \leq\pi^{-1}\sup_{z\in \gamma}\left[  \left\Vert\mathcal{T}\left(  z\right)
\right\Vert _{\infty}\right]  \int_{\gamma}\frac{1}{\left\vert z\right\vert
}\left\Vert \Gamma^{1/2}R^{1/2}(z)\right\Vert _{\infty
}^{2}dz\mathbf{1}_{\mathcal J_n}.
\end{align*}

For $z\in \text{supp}(\gamma)$, the eigenvalues of the operator $\Gamma^{1/2}R^{1/2}(z)$ are\\ $\left\{\lambda_j^{1/2}(z-\lambda_j)^{-1/2}, j\geq 1\right\}$, then 
\[\left\Vert\Gamma^{1/2}R^{1/2}(z)\right\Vert_\infty^2=\sup_{j\geq 1}\left\{\frac{\lambda_j}{|z-\lambda_j|}\right\},\]
and 
\begin{eqnarray*}
\left|\int_{\gamma}\frac{1}{\left\vert z\right\vert
}\left\Vert \Gamma^{1/2}R^{1/2}(z)\right\Vert _{\infty
}^{2}dz\right|
&\leq& C+2\int_0^{2 \lambda_1/\delta_{N_n}}\frac{du}{1+u^2}\leq C',
\end{eqnarray*} 
where $C$ and $C'$ are independent of $n$ and the last inequality comes from the fact that in the exponential case, there exists a constant $c>0$ such that $\delta_{N_n}/\lambda_{N_n}\geq c$. Then by lemmas~\ref{pbJn} and \ref{trick1}: 
\begin{eqnarray*}
\Pb(\mathcal J_n\cap\left\{\|\Gamma^{1/2}(\Gamma_n^\dag-\Gamma^\dag)\Gamma^{1/2}\|_\infty>\rho_0-1\right\})\leq \Pb\left(C'\sup_{z\in \text{supp}(\gamma)}\left[  \left\Vert\mathcal{T}\left(  z\right)
\right\Vert _{\infty}\right] >\pi(\rho_0-1)\right)\leq 2 \exp\left(-c^*\frac{n}{N_n^2\ln^2N_n}\right)
\end{eqnarray*}
with $c^*$ independent of $n$. The result comes from the fact that $N_n\leq 20 \sqrt{n/\ln^3 n}$.

\paragraph{Polynomial decrease : $\lambda_j\asymp j^{-a}$, $a>1$}
Denote by $T_1$ and $T_2$ the two terms of Equation~(\ref{forme2}) i.e.
\begin{eqnarray*}
T_1&=&\mathbf{1}_{\mathcal J_n}\frac{1}{2i\pi}\int_{\gamma}\frac{1}{z}\Gamma^{1/2}R(z)\left(  \Gamma_{n}-\Gamma\right)  R(z)
\Gamma^{1/2}dz ,\\
T_2&=&\mathbf{1}_{\mathcal J_n}\frac{1}{2i\pi}\int_{\gamma}\frac{1}{z}\Gamma^{1/2}R^{1/2}(z)\left[I-\mathcal{T}(z)\right]^{-1}\mathcal{T}^2(z)R^{1/2}(z)\Gamma^{1/2}dz.
\end{eqnarray*}

First we control $T_2$, the proof in the exponential case leads us to:
\begin{eqnarray*}
 \left\Vert T_2\right\Vert _{\infty}\leq \pi^{-1}\sup_{z\in \gamma}\left[  \left\Vert \mathcal{T}\left(  z\right)
\right\Vert _{\infty}^{2}\right]  \int_{\gamma}\frac{1
}{\left\vert z\right\vert }\left\Vert \Gamma^{1/2}R^{1/2}(z)\right\Vert _{\infty
}^{2}dz,
\end{eqnarray*}
and
\begin{eqnarray*}
\int_{\gamma}\frac{1}{\left\vert z\right\vert
}\left\Vert \Gamma^{1/2}R^{1/2}(z)\right\Vert _{\infty
}^{2}dz&\leq& C+\int_0^{2\lambda_1/\delta_{N_n}}\frac{\lambda_{N_n}du}{\sqrt{(\lambda_{N_n}-\delta_{N_n})^2+\delta_{N_n}^2u^2}\sqrt{1+u^2}}\\
&&\hspace{-5cm}\leq C+\int_0^1\frac{ \lambda_{N_n}}{\lambda_{N_n}- \delta_{N_n}}du+\int_1^{2\lambda_1/\delta_{N_n}}\frac{\lambda_{N_n}du}{\sqrt{(\lambda_{N_n}-\delta_{N_n})^2+\delta_{N_n}^2u^2}\sqrt{1+u^2}}\\
&&\hspace{-5cm}\leq 1+C+\int_1^{2\lambda_1/\delta_{N_n}}\frac{du}{\sqrt{1+u^2}}\leq C' \ln(N_n),
\end{eqnarray*}
with $C,C'>0$ independent of $n$.
Then lemmas~\ref{pbJn} and \ref{trick1} and the fact that  $N_n\leq 20 \sqrt{n/\ln^3 n}$ lead us to
\begin{eqnarray*}\Pb(\|T_2\|_\infty>(\rho_0-1)/2)\leq \Pb\left(C'\ln(N_n)\sup_{z\in \gamma}\left[  \left\Vert\mathcal{T}\left(  z\right)
\right\Vert _{\infty}^2\right] >\pi(\rho_0-1)/2\right)
\leq\exp\left(-c^{**}\frac{n}{N_n^2\ln^4(N_n)}\right)\leq C'' n^{-6},\end{eqnarray*}
with $c^**$ and $C''$ independent of $n$

Now, we can calculate explicitly the term $T_1$
\begin{equation*}
T_1=\mathbf{1}_{\mathcal J_n}\frac{1}{2i\pi}\int_{\gamma}\sum_{j,k\geq 1}\frac{\sqrt{\lambda_k}\sqrt{\lambda_j}}{z(z-\lambda_k)(z-\lambda_j)}\pi_k(\Gamma_n-\Gamma)\pi_jdz.
\end{equation*}
By the Residue Theorem
\[\frac{1}{2i\pi} \int_{\gamma}\frac{dz}{z(z-\lambda_k)(z-\lambda_j)}=\left\{\begin{array}{cl}
-\frac{1}{\lambda_j\lambda_k} &\text{ if }j\neq k, j\leq N_n\text{ and }k\leq N_n,\\
\frac{1}{\lambda_j(\lambda_j-\lambda_k)} &\text{if } j\leq N_n<k,\\
\frac{1}{\lambda_k(\lambda_k-\lambda_j)} &\text{if } k\leq N_n<j,\\
0&\text{otherwise}.
\end{array}\right.\]
Then
\begin{eqnarray*}
T_1&=&-\mathbf 1_{{\mathcal J}_n}\sum_{\substack{j,k=1\\j\neq k}}^{N_n}\frac{1}{\sqrt{\lambda_j\lambda_k}}\pi_k(\Gamma_n-\Gamma)\pi_j+\sum_{j=1}^{N_n}\sum_{k>N_n}\frac{\sqrt{\lambda_k}}{\sqrt{\lambda_j}(\lambda_j-\lambda_k)}\pi_k(\Gamma_n-\Gamma)\pi_j\\
&&+\sum_{j>N_n}\sum_{k=1}^{N_n}\frac{\sqrt{\lambda_j}}{\sqrt{\lambda_k}(\lambda_k-\lambda_j)}\pi_k(\Gamma_n-\Gamma)\pi_j\\
&=&-\mathbf 1_{{\mathcal J}_n}\sum_{\substack{j,k=1\\j\neq k}}^{N_n}\frac{\pi_k}{\sqrt{\lambda_k}}\Gamma_n\frac{\pi_j}{\sqrt{\lambda_j}}+\sum_{j=1}^{N_n}\left(s_j\Gamma_n\frac{\pi_j}{\sqrt{\lambda_j}}+\frac{\pi_j}{\sqrt{\lambda_j}} \Gamma_ns_j\right)=T_1'+T_1'',
\end{eqnarray*}
where $s_j:=\sum_{k>N_n}\frac{\sqrt{\lambda_k}}{\lambda_j-\lambda_k}\pi_k$. The term $\Gamma$ disappears because $\pi_j\Gamma\pi_k=0$ if $j\neq k$.

We control separately the operators $T_1'$ and $T_1''$. We have:
\[\|T_1'\|_\infty^2\leq\sum_{p,q\geq 1}<T_1'\psi_p,\psi_q>^2=\sum_{\substack{p,q=1\\p\neq q}}^{N_n}\left(\frac{1}{n}\sum_{i=1}^n\xi_p^{(i)}\xi_q^{(i)}\right)^2,\]
where we recall that $\xi^{(i)}_p=<X_i,\psi_p>/\sqrt{\lambda_p}$.
Then
\begin{eqnarray*}
\Pb(\|T_1'\|_\infty>(\rho_0-1)/4)\leq \Pb\left(\sum_{\substack{p,q=1\\p\neq q}}^{N_n}\left(\frac{1}{n}\sum_{i=1}^n\xi_p^{(i)}\xi_q^{(i)}\right)^2>\frac{(\rho_0-1)^2}{16}\right)\leq\sum_{\substack{p,q=1\\p\neq q}}^{N_n}\Pb\left(\left|\frac{1}{n}\sum_{i=1}^n\xi_p^{(i)}\xi_q^{(i)}\right|>\frac{\rho_0-1}{4N_n}\right).
\end{eqnarray*}
For all $p\neq q$, the sequence of random variables $\left\{\xi_p^{(i)}\xi_q^{(i)}, i=1,...,n\right\}$ is independent and centred and by assumptions~\textbf{H2} and \textbf{H3},
\[\E\left[|\xi_p\xi_q|^{m}\right]\leq m!b^{m-1},\]
then Lemma~\ref{bernstein} and the condition $N_n\leq 20\sqrt{n/\ln^3 n}$ implies
\[\Pb\left(\|T_1'\|_\infty>\frac{\rho_0-1}{4}\right)\leq 2N_n^2\exp\left(-C_1'\frac{n}{N_n^2}\right)\leq C_2' n^{-6},\]
with $C_1'=\frac{(\rho_0-1)^2}{32(2b+(\rho_0-1)/4)}$ and $C_2'$ depends only on $b$ and $\rho_0$.

We deal now with the operator $T_1''$, we can rewrite it like an array of independent random variables with values in $\mathcal H$, the set of the Hilbert-Schmidt operators of $\Ld$ equipped with the usual norm $\|T\|_{HS}^2=\sum_{p,q\geq 1}<T\psi_p,\psi_q>^2$, i.e. 
\[T_1''=\frac{1}{n}\sum_{i=1}^n\sum_{j=1}^{N_n}\left(s_jX_i\otimes\frac{\pi_j}{\sqrt{\lambda_j}}X_i+\frac{\pi_j}{\sqrt{\lambda_j}}X_i\otimes s_jX_i\right)=\frac1n\sum_{i=1}^nZ_i,\]
where, for all $f,g,h\in\Ld$, $f\otimes g\ h:=<f,h>g$.
 In order to apply the exponential inequality for centred Hilbert valued random variable given in Bosq~\cite[Theorem 2.5]{bosq_linear_2000}, we have to find two constants $B$ and $c$ such that \[\E\left[\|Z_i\|^m_{HS}\right]\leq (m!/2)B^2c^{m-2}.\]
We compute first $\|Z_i\|^2_{HS}$
\begin{equation*}
\|Z_i\|_{HS}^2\leq2\sum_{p,q\geq 1}<\sum_{j=1}^{N_n}\left(\frac{\pi_j}{\sqrt{\lambda_j}}X_i\otimes s_jX_i\right)\psi_p,\psi_q>^2=\sum_{p=1}^{N_n}\sum_{q>N_n}\frac{\lambda_q^2}{(\lambda_p-\lambda_q)^2}\left(\xi_p^{(i)}\xi_q^{(i)}\right)^2.
\end{equation*}
Now, by assumptions~\textbf{H2} and \textbf{H3},
\begin{eqnarray*}
\E\left[\|Z_i\|_{HS}^m\right]&\leq &\E\left[\|Z_i\|_{HS}^{2m}\right]^{1/2}=\left(\sum_{p_1,...,p_m=1}^{N_n}\sum_{q_1,...,q_m>N_n}\prod_{j=1}^m\frac{\lambda_{q_j}^2}{(\lambda_{p_j}-\lambda_{q_j})^2}\E\left[\left(\xi_{p_1}^{(i)}...\xi_{p_m}^{(i)}\right)^2\right]\E\left[\left(\xi_{q_1}^{(i)}...\xi_{q_m}^{(i)}\right)^2\right]\right)^{1/2}\\
&\leq &m!b^m\left(\sum_{p=1}^{N_n}\sum_{q>N_n}\frac{\lambda_q^2}{(\lambda_p-\lambda_q)^2}\right)^{m/2}.
\end{eqnarray*}
We apply then Theorem 2.5 of Bosq~\cite{bosq_linear_2000} with $B^2=2b^2\sum_{p=1}^{N_n}\sum_{q>N_n}\frac{\lambda_q^2}{(\lambda_p-\lambda_q)^2}$ and\\ $c=b\left(\sum_{p=1}^{N_n}\sum_{q>N_n}\frac{\lambda_q^2}{(\lambda_p-\lambda_q)^2}\right)^{1/2}$ and obtain, with the condition $N_n\leq 20\sqrt{n/\ln^3 n}$,
\[\Pb\left(\|T_1''\|_\infty>\frac{\rho_0-1}{4}\right)\leq \Pb\left(\|T_1''\|_{HS}>\frac{\rho_0-1}{4}\right)\leq2\exp\left(-C_1''\frac{n}{N_n}\right)\leq C_2''n^{-6} ;\]
where $C_1'':=(\rho_0-1)^2/(2\delta b^2+\sqrt{2\delta}b(\rho_0-1)/4)$, $C_2''$ depends only on $\rho_0$ and $\delta$
where $\delta>0$ depends only on the sequence $(\lambda_j)_{j\geq 1}$ and verifies, for all $p\leq N_n$,
\[\sum_{q>N_n}\frac{\lambda_q^2}{(\lambda_p-\lambda_q)^2}\leq N_n\delta/2.\]

\end{proof}

\section{Control of $\penh$ in the unknown variance case}
\label{subsec:penh}
\begin{lemma}\label{pen}
Under Assumption~\textbf{H1}, set $\kappa:=2\theta(1+2\delta)$, we have, for all $m\in\widehat{\Mnr}$
\begin{equation}\label{hpenmpen}\E_\mathbf{X}(\penh(m)-\pen^{(uv)}(m))\mathbf 1_G]\leq\kappa\frac{m}{n}\|\beta-\hat\Pi_m\beta\|_{\Gamma_n}^2,
\end{equation}
and
\begin{eqnarray}\EX[(\pen^{(uv)}(\hat m^{(uv)})-\penh(\hat m^{(uv)}))]\leq\frac{\kappa}{n} \mathbb E_\mathbf{X}[2\hmuv\nu_n(\hat\Pi_m\beta-\hat\beta_{\hmuv})]+\frac{\kappa \hat N_n}{n\sqrt n}\left(\sqrt{\Var(\varepsilon^2)}+2\|\beta\|_\Gamma\sigma\right).\end{eqnarray}
\end{lemma}
\begin{proof}
By definitions of $\hat\sigma_m^2=\gamma_n(\hat\beta_m)$ and $\hat\beta_m$, we have $\hat\sigma_m^2\leq\gamma_n(\hat\Pi_m\beta)$, then
\[\EX[(\widehat{\text{pen}}(m)-\penuv( m))]=\kappa\frac mn\EX[(\hat\sigma^2_m-\sigma^2)]\leq\kappa\frac mn\EX[(\gamma_n(\hat\Pi_m\beta)-\sigma^2)].\]

Now, by independence of $\varepsilon_i$ with $<\beta-\hat\Pi_m\beta,X_i>$,
\begin{eqnarray*}
\EX[(\gamma_n(\hat\Pi_m\beta)-\sigma^2)]&=&\EX\left[\left(\frac{1}{n}\sum_{i=1}^n(Y_i-<\hat\Pi_m\beta,X_i>)^2-\sigma^2\right)\right]\\
&=&\EX\left[\left(\frac{1}{n}\sum_{i=1}^n\varepsilon_i^2-2\varepsilon_i<\beta-\hat\Pi_m\beta,X_i>+<\beta-\hat\Pi_m\beta,X_i>^2-\sigma^2\right)\right]\\
&=&\EX[<\beta-\hat\Pi_m\beta,X_i>^2]=\EX[\|\beta-\hat\Pi_m\beta\|_{\Gamma_n}^2],
\end{eqnarray*}
and Equation~(\ref{hpenmpen}) follows.

Likewise: 
 \begin{eqnarray*}
 \EX[ (\penuv(\hmuv)-\widehat{\text{pen}}(\hmuv))]&=&\frac{\kappa}{n}\EX[\hmuv(\sigma^2-\hat\sigma^2_{\hmuv})]\\
 &\hspace{-11cm}=&\hspace{-5.5cm}\frac{\kappa}{n}(\EX[\hmuv(\sigma^2-\widetilde\sigma^2)]-\EX[\hmuv\|\beta-\hat\beta_{\hmuv}\|_{\Gamma_n}^2]+2\EX[\hmuv\nu_n(\beta-\hat\beta_{\hmuv})]\\
 &\hspace{-11cm}\leq&\hspace{-5.5cm}\frac{\kappa}{n}(\EX[\hmuv(\sigma^2-\widetilde\sigma^2)]+2\EX[\hmuv\nu_n(\beta-\hat\beta_{\hmuv})])\\
 &\hspace{-11cm}\leq&\hspace{-5.5cm}\frac{\kappa}{n}(\EX[\hmuv(\sigma^2-\widetilde\sigma^2)]+2\EX[\hmuv\nu_n(\beta-\hat\Pi_m\beta)+\hmuv\nu_n(\hat\Pi_m\beta-\hat\beta_{\hmuv})])
 \end{eqnarray*}
 with $\widetilde\sigma^2:=\frac{1}{n}	\sum_{i=1}^n\varepsilon_i^2$.
 By Cauchy-Schwarz's Inequality
 \begin{eqnarray*}
 \EX[\hmuv(\sigma^2-\widetilde\sigma^2)]\leq \hat N_n\EX[(\sigma^2-\widetilde\sigma^2)^2]^{1/2}=\hat N_n\left(\frac{1}{n^2}\sum_{i,j=1}^n\EX[(\varepsilon_i^2-\sigma^2)(\varepsilon_j^2-\sigma^2)]\right)^{1/2}
 =\hat N_n\sqrt{\frac{\Var(\varepsilon_i^2)}{n}},
 \end{eqnarray*}
 and, since the $\varepsilon_i$'s are independent of the $X_i$'s and by consequence of $\hat\Pi_m$, we have:
 \begin{eqnarray*}
 \EX[\hat m\nu_n(\beta-\hat\Pi_m\beta)]&\leq& \hat N_n\EX[\nu_n^2(\beta-\hat\Pi_m\beta)]^{1/2}\\
 &\leq& \frac{\hat N_n}{n}\left(\sum_{i_1,i_2=1}^n\EX[\varepsilon_{i_1}\varepsilon_{i_2}<\beta-\hat\Pi_m \beta,X_{i_1}><\beta-\hat\Pi_m \beta,X_{i_2}>]\right)^{1/2}\\
 &\leq&\frac{\hat N_n}{\sqrt n}\sigma\E[\|\beta-\hat\Pi_m\beta\|_{\Gamma_n}^2]^{1/2}\leq\frac{\hat N_n}{\sqrt n}\sigma\|\beta\|_\Gamma .
 \end{eqnarray*}
\end{proof}

\end{appendices}
\bibliography{biblio}
\bibliographystyle{siam}
\end{document}